\pdfoutput=1
\documentclass[]{amsart}
\PassOptionsToPackage{final}{graphicx}
\PassOptionsToPackage{final,hidelinks}{hyperref}
\usepackage{ifdraft}
\usepackage{nima}
\usepackage{bbm}
\usepackage{tikz-cd}
\ifdraft{\usepackage{screen}}{}
\usepackage{xnotes}
\usepackage{amsaddr}

\DeclareMathOperator{\CAT}{CAT}
\DeclareMathOperator{\BS}{BS}

\DeclareMathOperator{\diam}{diam}

\let\Cay\undefined
\DeclareMathOperator{\Cay}{Cay}

\theoremstyle{definition}
\newtheorem*{cons*}{Construction}
\newtheorem{cons}[thmctr]{Construction}
\newcommand{\conslabel}[1]{\label{cons:#1}}
\newcommand{\consref}[1]{\ref{cons:#1}}
\newcommand{\Consref}[1]{Construction~\consref{#1}}

\newcommand{\pitmref}[1]{(\itmref{#1})}
\newcommand{\Lemitmref}[2]{\Lemref{#1}\pitmref{#2}}

\renewcommand{\sth}{:}
\newcommand{\actson}{\curvearrowright}

\newcommand{\eqnlabel}[1]{\label{eqn:#1}}
\newcommand{\eqnref}[1]{\ref{eqn:#1}}
\newcommand{\peqnref}[1]{(\eqnref{#1})}

\newcommand{\pc}[1]{P^{\circ}_{#1}}

\title{Shortcut Graphs and Groups}
\author{Nima Hoda}
\address{D{\' e}partement de math{\' e}matiques et applications \\
  {\' E}cole normale sup{\' e}rieure, 45 rue d'Ulm, 75005 Paris, France
  \\ \ \\ Instytut Matematyczny,
  Uniwersytet Wroc\l awski\\
  pl.\ Grun\-wal\-dzki 2/4,
  50--384 Wroc{\l}aw, Poland}

\thanks{This work was supported in part by an NSERC Postgraduate
  Scholarship, a Pelletier Fellowship, a Graduate Mobility Award, an
  ISM Scholarship, the NSERC grant of Daniel T. Wise and the ERC grant
  GroIsRan.}

\email{nima.hoda@mail.mcgill.ca}
\date{\today}


\begin{document}

\begin{abstract}
  We introduce shortcut graphs and groups.
  \texorpdfstring{\newterm{Shortcut graphs}}{Shortcut graphs} are
  graphs in which cycles cannot embed without metric distortion.
  \texorpdfstring{\newterm{Shortcut groups}}{Shortcut groups} are
  groups which act properly and cocompactly on shortcut graphs.  These
  notions unify a surprisingly broad family of graphs and groups of
  interest in geometric group theory and metric graph theory,
  including: the \texorpdfstring{$1$}{1}-skeletons of systolic and
  quadric complexes (in particular finitely presented C(6) and
  C(4)-T(4) small cancellation groups),
  \texorpdfstring{$1$}{1}-skeletons of finite dimensional
  \texorpdfstring{$\CAT(0)$}{CAT(0)} cube complexes, hyperbolic
  graphs, standard Cayley graphs of finitely generated Coxeter groups
  and the standard Cayley graph of the Baumslag-Solitar group
  \texorpdfstring{$\BS(1,2)$}{BS(1,2)}.  Most of these examples
  satisfy a strong form of the shortcut property.

  The shortcut properties also have important geometric group
  theoretic consequences.  We show that shortcut groups are finitely
  presented and have exponential isoperimetric and isodiametric
  functions.  We show that groups satisfying the strong form of the
  shortcut property have polynomial isoperimetric and isodiametric
  functions.
\end{abstract}

\maketitle

\tableofcontents

\section{Introduction}

A main current in geometric group theory is the study of groups acting
on spaces satisfying various kinds of combinatorial nonpositive
curvature properties.  These spaces are typically associated with
graphs having nice metric properties.  For example, the $1$-skeletons of
$\CAT(0)$ cube complexes \cite{Roller:1998, Gerasimov:1998,
  Chepoi:2000}, systolic complexes \cite{Chepoi:2000} and quadric
complexes \cite{Hoda:2017}, all of which arose independently in the
geometric group theory and metric graph theory literature
\cite{Gromov:1987, Januszkiewicz:2006, Haglund:2003, Hoda:2017,
  Avann:1961, Nebesky:1971, Klavzar:1999, Soltan:1983, Bandelt:1988}.
In this paper we introduce a metric property that captures an aspect
of nonpositive curvature that is shared by these graphs and a
surprisingly large number of other graphs of importance in group
theory and metric graph theory.

A \defterm{shortcut graph} is a graph $\Gamma$ for which there is a
bound on the lengths of its isometrically embedded cycles.  A
\defterm{strongly shortcut graph} $\Gamma$ has a bound on the lengths
of its $K$-bilipschitz cycles, for some fixed $K > 1$.  A group is
\defterm{(strongly) shortcut} if it acts properly and cocompactly on a
(strongly) shortcut graph.

Initial motivation for the study of shortcut graphs and groups arose
from systolic and quadric complexes.  Chepoi characterized systolic
complexes as those flag simplicial complexes whose $1$-skeletons
contain isometrically embedded cycles only of length three (i.e. their
only isometrically embedded cycles are triangles) \cite{Chepoi:2000}.
We similarly characterized quadric complexes as those square-flag
square complexes whose $1$-skeletons contain isometrically embedded
cycles only of length four \cite{Hoda:2017}.  Hence the $1$-skeletons
of systolic and quadric complexes are shortcut.  In particular,
systolic and quadric groups, and thus finitely presented C(6) and
C(4)-T(4) small cancellation groups, are shortcut \cite{Wise:2003,
  Hoda:2017}.  As we will see, many prominent classes of graphs and
groups satisfy the shortcut property.

\subsection{Summary of results}

The following theorems summarize our main results.

\begin{mainthm}[\Corref{fundgroup}]
  \mainthmlabel{finpres} Shortcut groups are finitely presented.
\end{mainthm}

\begin{mainthm}[\Thmref{isoper}, \Thmref{isodiam}]
  \mainthmlabel{filling} Shortcut graphs and groups have exponential
  isoperimetric and isodiametric functions.  Strongly shortcut graphs
  and groups have polynomial isoperimetric and isodiametric functions.
  Consequently, shortcut groups have decidable word problem.
\end{mainthm}



\begin{mainthm}[\Thmref{hypgraph}, \Corref{cccs}, \Thmref{sys_quad_ss}, \Thmref{coxs}, \Thmref{cayz}, \Thmref{cayzz}]
  \mainthmlabel{examps} The following classes of graphs are strongly
  shortcut.
  \begin{itemize}
  \item Hyperbolic graphs
  \item $1$-skeletons of finite dimensional $\CAT(0)$ cube complexes
  \item $1$-skeletons of systolic complexes
  \item $1$-skeletons of quadric complexes
  \item Standard Cayley graphs of finitely generated Coxeter groups
  \item All Cayley graphs of $\Z$ and $\Z^2$
  \end{itemize}
  In particular, hyperbolic groups, cocompactly cubulated groups,
  systolic groups, quadric groups, $C(6)$ small cancellation groups,
  $C(4)$-$T(4)$ small cancellation groups and Coxeter groups are all
  strongly shortcut.
\end{mainthm}

\begin{mainthm}[\Thmref{bss}, \Thmref{bsns}]
  \mainthmlabel{bsgp} The Baumslag-Solitar group $\BS(1,2)$ is
  shortcut but not strongly shortcut.  Moreover, $\BS(1,2)$ has a
  Cayley graph which is shortcut and a Cayley graph which is not
  shortcut.
\end{mainthm}

\subsection{Structure of the paper}

In \Secref{defs} we present the main definitions of the paper.  In
\Secref{basicprops} we prove basic properties of shortcut graphs and
groups.  In \Secref{fill} we construct disk diagrams for shortcut
graphs and study their filling invariants.  In \Secref{combin} we
prove that products of (strongly) shortcut graphs are (strongly)
shortcut and that finite graphs of (strongly) shortcut groups with
finite edge groups are (strongly) shortcut.  In \Secref{examp} we
prove that various classes of graphs and groups are strongly shortcut.
In \Secref{bsgp} we study the shortcut property in Cayley graphs of
$\BS(1,2)$.

\subsection*{Acknowledgements}
The author would like to acknowledge Daniel T. Wise for many
invaluable discussions throughout the course of this research and for
his encouragement from its outset.  Thanks to Piotr Przytycki for a
very detailed list of comments and corrections for a version of this
paper included in the author's thesis.  Finally, thanks to the
reviewer for a very thorough and detailed review.

\section{Definitions}
\seclabel{defs}

A \defterm{graph} $\Gamma$ is a $1$-dimensional polyhedral complex
whose edges are isometric to $[0,1] \subset \R$.  In this way each
graph is equipped with both the structure of a cellular complex with
edges and vertices and the structure of a geodesic metric giving a
distance between any pair of its points.  A \defterm{combinatorial
  map} of graphs is a continuous map $\Gamma_1 \to \Gamma_2$ in which
each vertex of $\Gamma_1$ maps onto a vertex of $\Gamma_2$ and each
closed edge of $\Gamma_1$ maps onto a vertex or closed edge of
$\Gamma_2$.  A combinatorial map is \defterm{degenerate} if some
closed edge maps onto a vertex.  Otherwise it is
\defterm{nondegenerate}.

\subsection{Isometric and almost isometric cycles}

Let $\Gamma$ be a graph.  A \defterm{combinatorial path in $\Gamma$}
is a nondegenerate combinatorial map $P \to \Gamma$ from a graph $P$
that is homeomorphic to a compact interval of $\R$.  A \defterm{cycle}
$C$ is a graph homeomorphic to a circle.  A \defterm{cycle in
  $\Gamma$} is a nondegenerate combinatorial map $C \to \Gamma$ from a
cycle $C$.  The \defterm{length} of a path or cycle, denoted $|P|$ or
$|C|$, is the number of its edges.

A cycle $f \colon C \to \Gamma$ is \defterm{isometric} if it is an
isometric embedding.  \Corref{isomisomemb} below shows that $f$ is
isometric if and only if
\[ d_{\Gamma}\bigl(f(p),f(q)\bigr) \ge \frac{|C|}{2} \] for every
antipodal pair of points $p,q \in C$.  With this in mind we give the
following definition.  A cycle $f \colon C \to \Gamma$ is
\defterm{$\xi$-almost isometric}, for $\xi \in (0,1)$, if
\[ d_{\Gamma}\bigl(f(p),f(q)\bigr) \ge \xi \frac{|C|}{2} \] for every
antipodal pair of points $p,q \in C$.  One may imagine that if $f$ is
not isometric then there is a ``shortcut'' in $\Gamma$ between a pair
of its antipodes and if $f$ is not $\xi$-almost isometric then there
is such a ``shortcut'' which reduces the distance by a constant
factor.

\subsection{Shortcut graphs and groups}
A connected simplicial graph $\Gamma$ is \defterm{shortcut} if, for
some $\theta \in \N$, every isometric cycle $C \to \Gamma$ has length
$|C| \le \theta$.  A connected simplicial graph $\Gamma$ is
\defterm{strongly shortcut} if, for some $\theta \in \N$ and some
$\xi \in (0,1)$, every $\xi$-almost isometric cycle $C \to \Gamma$ has
length $|C| \le \theta$.

\begin{rmk}
  It follows immediately from the definitions that if $\Gamma'$ is an
  isometrically embedded subgraph of a (strongly) shortcut graph
  $\Gamma$ then $\Gamma'$ is also (strongly) shortcut.
\end{rmk}

\Propref{bilip} below says that $\Gamma$ is strongly shortcut if and
only if there is a $K > 1$ and a bound on the lengths of the
$K$-bilipschitz cycles of $\Gamma$.  Since a $1$-bilischitz map is the
same thing as an isometric embedding, we can thus define both
properties together as follows: $\Gamma$ is shortcut if there is a
$K \ge 1$ and a bound on the lengths of the $K$-bilipschitz cycles of
$\Gamma$; if this $K$ can be chosen strictly greater than $1$ then
$\Gamma$ is strongly shortcut.

A group $G$ is \defterm{(strongly) shortcut} if it acts properly and
cocompactly on a (strongly) shortcut graph.

\section{Basic properties}
\seclabel{basicprops}

In this section we prove some basic properties of (strongly) shortcut
graphs and groups.  In particular, we prove that the strong shortcut
property is equivalent to the existence of a $K > 1$ for which there
is a bound on the lengths of the $K$-bilipschitz cycles.  This
characterization is particularly interesting in light of a
characterization of Hume and MacKay of a hyperbolic graph as a graph
for which there is a bound on the lengths of its $18$-bilipschitz
cycles \cite{hume_mackay:poorly_connected:2020}.  We also show that a
(strongly) shortcut group acts freely and cocompactly on a (strongly)
shortcut graph.

\begin{prop}
  \proplabel{pqna} Let $\Gamma$ be a graph and let $\bar \xi \in (0,1]$.  A
  cycle $f \colon C \to \Gamma$ satisfies
  \[ d_{\Gamma}\bigl(f(p),f(q)\bigr) \ge \bar \xi \frac{|C|}{2} \] for
  every antipodal pair of points $p,q \in C$ if and only if
  \[ d_{\Gamma}\bigl(f(p),f(q)\bigr) \ge d_C(p,q) -
    (1-\bar \xi)\frac{|C|}{2} \] for every pair of points $p,q \in C$.
\end{prop}
\begin{proof}
  The ``if'' part follows by applying the inequality to each antipodal
  pair of $p$ and $q$.  To prove the ``only if'' part, let
  $p,q \in C$.  Let $p'$ be the antipode of $p$.  Then, since $f$ is
  $1$-Lipschitz, we have
  \begin{align*}
    d_{\Gamma}\bigl(f(p),f(p')\bigr) &\le
    d_{\Gamma}\bigl(f(p),f(q)\bigr) + d_{\Gamma}\bigl(f(q),f(p')\bigr) \\
    &\le d_{\Gamma}\bigl(f(p),f(q)\bigr) + d_C(q,p') \\ &=
    d_{\Gamma}\bigl(f(p),f(q)\bigr) + d_C(p,p') - d_C(p,q) \\ &=
    d_{\Gamma}\bigl(f(p),f(q)\bigr) + \frac{|C|}{2} - d_C(p,q)
  \end{align*}
  but $\bar \xi\frac{|C|}{2} \le d_{\Gamma}\bigl(f(p),f(p')\bigr)$ and so
  we have
  $d_{\Gamma}\bigl(f(p),f(q)\bigr) \ge d_C(p,q) -
  (1-\bar \xi)\frac{|C|}{2}$.
\end{proof}

\begin{cor}
  \corlabel{isomisomemb} Let $\Gamma$ be a graph and let
  $f \colon C \to \Gamma$ be a cycle.  Then $f$ is isometric if and
  only if \[ d_{\Gamma}\bigl(f(p),f(q)\bigr) \ge \frac{|C|}{2} \] for
  every antipodal pair of points $p,q \in C$.
\end{cor}
\begin{proof}
  This follows from the fact that $f$ is $1$-Lipschitz and by applying
  \Propref{pqna} with $\bar \xi = 1$.
\end{proof}

\begin{prop}
  \proplabel{aadist} Let $\Gamma$ be a graph and let
  $f \colon C \to \Gamma$ be a cycle in $\Gamma$ of length
  $|C| \ge 4$.  If $f$ is not isometric then
  \[ d_{\Gamma}\bigl(f(u),f(v)\bigr) < d_C(u,v) \] for some pair of
  vertices $u,v \in C^0$ with
  $d_C(u,v) \ge \Bigl\lfloor\frac{|C|}{2}\Bigr\rfloor - 1$.  If $f$ is
  not $\xi$-almost isometric, for some $\xi \in (0,1)$, then
  \[ d_{\Gamma}\bigl(f(u),f(v)\bigr) < \xi d_C(u,v) \] for some pair
  of vertices $u,v \in C^0$ with
  $d_C(u,v) \ge \Bigl\lfloor\frac{|C|}{2}\Bigr\rfloor - 1$.
\end{prop}
\begin{proof}
  Let $\bar \xi \in (0,1]$.  Suppose
  \[ d_{\Gamma}\bigl(f(p),f(q)\bigr) < \bar \xi \frac{|C|}{2} \] for
  some pair of antipodal points $p,q \in C$.  If $\bar \xi < 1$ then
  this is equivalent to $f$ not being $\bar \xi$-almost isometric and,
  otherwise, this is equivalent to $f$ not being isometric.

  Let $\alpha\colon [0,d] \to \Gamma$ be a geodesic from $f(p)$ to
  $f(q)$ where $d = d_{\Gamma}\bigl(f(p),f(q)\bigr)$.  If
  $\alpha^{-1}(\Gamma^0) = \emptyset$ then $f(p)$ and $f(q)$ are
  contained in the interior of some common edge $e$ of $\Gamma$.  Then
  the edges $e_1$ and $e_2$ of $C$ with $p \in e_1$ and $q \in e_2$
  map onto $e$.  Then there are endpoints $u \in e_1$ and $v \in e_2$
  such that $f(u) = f(v)$ and $d_C(p,u) + d_C(q,v) \le 1$.  So we have
  $d_{\Gamma}\bigl(f(u),f(v)\bigr) = 0$ and
  $d_C(u,v) \ge d_C(p,q) - d_C(p,u) - d_C(q,v) \ge \frac{|C|}{2} - 1$.
  
  Assume now that $p$ and $q$ do not map to the same edge in $\Gamma$.
  Let $\alpha^{-1}(\Gamma^0) = \{x_1,x_2, \ldots, x_k\}$ with
  $0 \le x_1 < x_2 < \cdots < x_k \le d$.  Then $\alpha|_{[0,x_1]}$
  and $\alpha|_{[x_k,d]}$ factor through $f$ so their images contain
  vertices $u,v \in C^0$ with $d_C(p,u) < 1$ and $d_C(q,v) < 1$ such
  that $f(u) = \alpha(x_1)$ and $f(v) = \alpha(x_k)$.  Moreover
  \begin{align*}
    d_{\Gamma}\bigl(f(u),f(v)\bigr)
    &= x_k - x_1 \\
    &= d - x_1 - (d - x_k) \\
    &= d - d_C(p,u) - d_C(q,v) \\
    &< \bar \xi d_C(p,q) - d_C(p,u) - d_C(q,v) \\
    &\le \bar \xi \bigl(d_C(p,q) - d_C(p,u) - d_C(q,v)\bigr) \\
    &\le \bar \xi d_C(u,v)
  \end{align*}
  and
  $d_C(u,v) \ge d_C(p,q) - d_C(p,u) - d_C(q,v) > \frac{|C|}{2} - 2$.
  Since $|C|$ and $d_C(u,v)$ are integers we obtain
  $d_C(u,v) \ge \Bigl\lfloor\frac{|C|}{2}\Bigr\rfloor - 1$
\end{proof}

\begin{rmk}
  In fact, if $f\colon C \to \Gamma$ is not isometric then we can
  improve the $u$ and $v$ obtained from \Propref{aadist} so that
  $d_C(u,v) \ge \Bigl\lfloor\frac{|C|}{2}\Bigr\rfloor$.  This does not
  hold for $f$ not $\xi$-almost isometric.  (Consider the quotient map
  from $C$ oriented and of even length that identifies two antipodal
  edges of $C$ in an orientation reversing way.)  In order to present
  a unified proof without additional case analysis, we give the weaker
  statement in \Propref{aadist}.
\end{rmk}

\begin{prop}
  \proplabel{bilip} Let $\Gamma$ be a graph.  Then $\Gamma$ is
  strongly shortcut if and only if there exists $K > 1$ such that
  there is a bound on the length of $K$-bilipschitz cycles of
  $\Gamma$.
\end{prop}
\begin{proof}
  If a cycle $f \colon C \to \Gamma$ is not $\xi$-almost isometric
  then, for some pair of antipodal points $p,q \in C$,
  \[ \xi d_C(p,q) = \xi \frac{|C|}{2} >
    d_{\Gamma}\bigl(f(p),f(q)\bigr) \] and so $f$ is not
  $\frac{1}{\xi}$-bilipschitz.  This proves the ``only if'' part of
  the proposition.
  
  To prove the ``if'' part of the proposition, suppose $\theta$ bounds
  the length of the $K$-bilipschitz cycles of $\Gamma$ where $K > 1$.
  Let $1 - \frac{(K-1)^3}{13K^2(K+1)} < \xi < 1$.  We will show that
  there is a bound on the lengths of the $\xi$-almost isometric cycles
  of $\Gamma$.  Let $f \colon C \to \Gamma$ be a $\xi$-almost
  isometric cycle of $\Gamma$.  We will define a sequence of paths
  $(P_i)$, a sequence of cycles $(C_i)_i$, a sequence of finite graphs
  $(\Gamma_i)_i$ and sequences of combinatorial maps as in the
  following commuting diagram. 
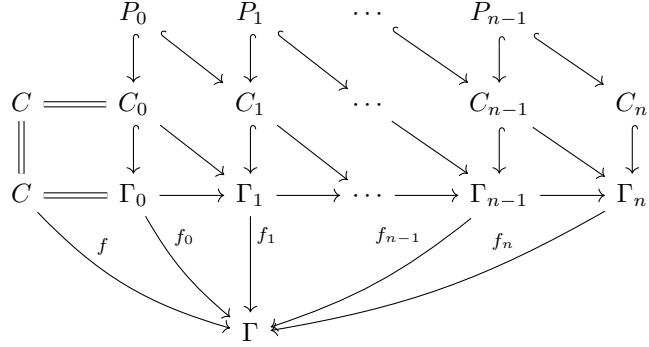
\[ \begin{tikzcd}
    &P_0 \ar[d,hook] \ar[dr] & P_1 \ar[d,hook] \ar[dr] & \cdots \ar[dr] & P_{n-1} \ar[d,hook] \ar[dr] \\
    C \ar[r,equal] \ar[d,equal] &C_0 \ar[dr] \ar[d,hook] & C_1 \ar[dr] \ar[d,hook] & \cdots \ar[dr] & C_{n-1} \ar[dr] \ar[d,hook] & C_n \ar[d,hook] \\
    C \ar[r,equal] \ar[ddrr,bend right=15,pos=0.3,"f"] &\Gamma_0 \ar[r] \ar[ddr,bend right=10,pos=0.3,"f_0"] & \Gamma_1 \ar[r] \ar[dd,pos=0.2,"f_1"] & \cdots \ar[r] & \Gamma_{n-1} \ar[ddll,bend left=10,pos=0.22,"f_{n-1}"'] \ar[r] & \Gamma_n \ar[ddlll,bend left=10,pos=0.25,"f_n"'] \\
    \\
    &&\Gamma
  \end{tikzcd} \]
  Where it makes sense, we will use the same notation to refer to
  points and subspaces as we do to refer to their images under maps.
  We begin with $\Gamma_0 = C_0 = C$ and $f_0 = f$.  Suppose we
  inductively have
  $C_i \hookrightarrow \Gamma_i \xrightarrow{f_i} \Gamma$.  If the
  composition of these maps is $K$-bilipschitz then we terminate the
  sequence with $n = i$.  Otherwise, let $u_i,v_i \in C_i^0$ be a
  furthest pair of vertices in $C_i$ for which
  $d_{\Gamma}\bigl(f_i(u_i),f_i(v_i)\bigr) < \frac{1}{K}
  d_{C_i}(u_i,v_i)$.
  Let $Q_i$ be a geodesic segment of $C_i$ between $u_i$ and $v_i$.
  Let $P_i$ be the closure of the complement of $Q_i$ and let
  $\pc{i}$ be the interior of $P_i$.  Let $R_i \to \Gamma$ be a
  geodesic from $f_i(u_i)$ to $f_i(v_i)$.  We obtain
  $f_{i+1} \colon \Gamma_{i+1} \to \Gamma$ from $f_i$ and
  $R_i \to \Gamma$ by identifying the endpoints of $R_i$ with
  $\{u_i,v_i\}$.  Let $C_{i+1} = P_i \cup R_i$ in $\Gamma_{i+1}$.  The
  sequence always terminates since $|C_i|$ is strictly decreasing.

  Our goal is to show that $\frac{|C_n|}{|C|}$ is uniformly bounded
  away from zero.  Thus we will show that if we have arbitrarily long
  $\xi$-almost isometric cycles then we must also have arbitrarily
  long $K$-bilipschitz cycles.
  
  For each $i$, the interior $\pc{i}$ embeds in both $C_i$ and
  $C_{i+1}$.  Let $\pc{0,j}$ be the limit of the diagram
  \[ \begin{tikzcd}
    \pc{0} \ar[d,hook] \ar[dr,hook] &
    \pc{1} \ar[d,hook] \ar[dr,hook] &
    \cdots \ar[d,hook] \ar[dr,hook] &
    \pc{j-1} \ar[d,hook] \ar[dr,hook] \\
    C_0 & C_1 & \cdots & C_{j-1} & C_j
  \end{tikzcd} \]
  in the category of topological spaces and continuous maps.
  Concretely, we have $\pc{0,0} = C_0$ and $\pc{0,1} = \pc{0}$ and
  $\pc{0,j} = \pc{0,j-1} \cap \pc{j-1}$ where the intersection is
  taken in $C_{j-1}$.  Thus we have the following commutative
  diagram of embeddings.
  \[ \begin{tikzcd}
    \pc{0,1} \ar[d,equal] &
    \pc{0,2} \ar[d,hook] \ar[l,hook] &
    \pc{0,3} \ar[d,hook] \ar[l,hook] &
    \cdots \ar[d,hook] \ar[l,hook] &
    \pc{0,n-1} \ar[d,hook] \ar[l,hook] &
    \pc{0,n} \ar[d,hook] \ar[l,hook] \\
    \pc{0} \ar[d,hook] \ar[dr,hook] &
    \pc{1} \ar[d,hook] \ar[dr,hook] &
    \pc{2} \ar[d,hook] \ar[dr,hook] &
    \cdots \ar[d,hook] \ar[dr,hook] &
    \pc{n-2} \ar[d,hook] \ar[dr,hook] &
    \pc{n-1} \ar[d,hook] \ar[dr,hook] \\
    C_0 & C_1 & C_2 & \cdots & C_{n-2} & C_{n-1} & C_n
  \end{tikzcd} \]
  We can think of $\pc{0,j}$ as original points of $C$ that are not
  replaced until at least step $j$ of the construction of the $C_i$,
  where the $i$th step of the construction refers to the operation of
  replacing $Q_i \to \Gamma$ with $R_i \to \Gamma$ in order to obtain
  $C_{i+1} \to \Gamma$ from $C_i \to \Gamma$.

  Fix $j$ and suppose $Q_i \subset \pc{0,i}$, for all $i < j$, where
  $\pc{0,i}$ is viewed as a subspace of $C_i$ via the inclusion
  $\pc{0,i} \hookrightarrow C_i$.  So, for $i<j$, we have lifts
  $Q_i \hookrightarrow \pc{0,i}$ as in the diagram
  \[ \begin{tikzcd}
    Q_0 \ar[d,hook] &
    Q_1 \ar[d,hook] &
    Q_2 \ar[d,hook] &
    \cdots \ar[d,hook] &
    Q_{j-2} \ar[d,hook] &
    Q_{j-1} \ar[d,hook] \\
    C &
    \pc{0,1} \ar[l,hook] &
    \pc{0,2} \ar[l,hook] &
    \cdots \ar[l,hook] &
    \pc{0,j-2} \ar[l,hook] &
    \pc{0,j-1} \ar[l,hook] &
    \pc{0,j} \ar[l,hook]
  \end{tikzcd} \]
  and we have disjoint unions $\pc{0,i} = Q_i \sqcup \pc{0,i+1}$.  So,
  for $i < j$, we may think of the $Q_i$ as subspaces of $C$.  The
  $Q_i$ are disjoint in $C$ and
  $C \setminus \pc{0,j} = \bigcup_{i=0}^{j-1}Q_i$.  Since $C_j$ is
  obtained from $C$ by replacing $Q_i$ with $R_i$, for each $i < j$,
  we see that the $R_i$, with $i < j$, embed disjointly in $C_j$ and
  the complement in $C_j$ of $\pc{0,j}$ is $\bigcup_{i=0}^{j-1}R_i$.
  Since $f$ is $\xi$-almost isometric we have
  \[ |Q_i| - (1 - \xi)\frac{|C|}{2} \le |R_i| < \frac{1}{K} |Q_i| \]
  for all $i < j$, by \Propref{pqna}.  Hence, if $i < j$ then we have
  the following inequality.
  \begin{equation}
    \eqnlabel{onebound}  K|R_i| < |Q_i| < (1-\xi)\biggl(\frac{K}{K-1}\biggr)
    \frac{|C|}{2}
  \end{equation}
  Moreover, we can find a pair of points $p,q$ in the closure of
  $C \setminus \bigl(\bigcup_{i=1}^{j-1}Q_i\bigr)$ at distance
  $d_C(p,q) \ge \frac{|C|}{2} - \frac{K(1-\xi)}{4(K-1)}|C|$.  Let
  $S_1$ and $S_2$ be the two segments of $C$ between $p$ and $q$.  If
  $I_1 = \{i < j \sth Q_i \subset S_1 \}$ then
  \begin{align*}
    d_{\Gamma}\bigl(f(p),f(q)\bigr)
    &\le |S_1| - \sum_{i \in I_1}|Q_i| + \sum_{i \in I_1}|R_i| \\
    &< |S_1| - \sum_{i \in I_1}|Q_i| + \frac{1}{K}\sum_{i \in I_1}|Q_i| \\
    &= |S_1| - \frac{K - 1}{K}\sum_{i \in I_1}|Q_i| \\
  \end{align*}
  and the same holds for $S_2$ and so, by \Propref{pqna},
  \begin{align*}
    &|C| - \frac{K-1}{K}\sum_{i<j}|Q_i| \\
    &\ge 2 d_{\Gamma}\bigl(f(p),f(q)\bigr) \\
    &\ge 2d_C(p,q) - (1-\xi)|C| \\
    &\ge |C| - \frac{K(1-\xi)}{2(K-1)}|C| - (1-\xi)|C|
  \end{align*}
  which gives us the following inequality.
  \begin{equation}
    \eqnlabel{sumbound}
    \sum_{i<j}|Q_i| \le (1-\xi)\biggl(\frac{K(3K-2)}{(K-1)^2}\biggr)\frac{|C|}{2}
  \end{equation}

  We will now prove that $Q_i \subset \pc{0,i}$ for $1 \le i < n$.
  For the sake of finding a contradiction, suppose $j \ge 1$ is the
  least integer with $Q_j \not\subset \pc{0,j}$.  As above we view the
  $R_i$ with $i < j$ as disjoint segments of $C_j$ with
  $C_j \setminus \bigcup_{i=0}^{j-1}R_i = \pc{0,j}$.  It is possible
  that, for some $i < j$, we have $Q_j \cap R_i \neq \emptyset$ in
  $C_j$ but $R_i \not\subset Q_j$.  This may happen for at most two
  $R_i$ since such $R_i$ must contain an endpoint of $Q_j$.  Let
  $Q_j^{-}$ be obtained from $Q_j$ by subtracting the interiors of any
  such $R_i$ and let $Q_j^{+} \to C_j$ extend
  $Q_j \hookrightarrow C_j$ so as to include full copies of any such
  $R_i$.  Let $Q^{-} \subset C$ be obtained from $Q_j^{-} \subset C_j$
  by replacing any $R_i \subset Q_j^{-}$, where $i < j$, with
  $Q_i \subset C$.  Let $Q^{+} \to C$ be obtained from
  $Q_j^{+} \to C_j$ by replacing any $R_i \hookrightarrow C_j$, where
  $i < j$, with $Q_i \hookrightarrow C$.  Let $R^{+} \to \Gamma$ be
  obtained from $Q_j^{+} \to \Gamma$ by replacing $Q_j \to \Gamma$
  with $R_j \to \Gamma$.  Then $R^{+} \to \Gamma$ and the composition
  $Q^{+} \to C \xrightarrow{f} \Gamma$ have the same endpoints in
  $\Gamma$ and we have
  \begin{align*}
    |R^{+}| &= |R_j| + |Q_j^{+} \setminus Q_j| \\
        &< \frac{1}{K}|Q_j| + |Q_j^{+} \setminus Q_j| \\
        &=\frac{1}{K}\bigl(|Q_j^{-}| + |Q_j \setminus Q_j^{-}|\bigr)
          + |Q_j^{+} \setminus Q_j| \\
        &\le\frac{1}{K}|Q_j^{-}| + |Q_j \setminus Q_j^{-}|
          + |Q_j^{+} \setminus Q_j| \\
        &=\frac{1}{K}|Q_j^{-}| + |Q_j^{+} \setminus Q_j^{-}| \\
        &\le \frac{1}{K}|Q_j^{-}| + \frac{1}{K}|Q^{+} \setminus Q^{-}| \\
        &= \frac{1}{K}|Q^{+}|
  \end{align*}
  where the final inequality follows from the fact that
  $Q_j^{+} \setminus Q_j^{-}$ consists of up to two copies of segments
  $R_i$ which are replaced with corresponding $Q_i$ in
  $Q^{+} \setminus Q^{-}$.  By assumption, $Q_j$ nontrivially
  intersects at least one $R_i$, with $i < j$.  Let $m$ be minimal
  such that $Q_j$ nontrivially intersects $R_m$.  Since $Q_j$ is not
  equal to this $R_m$ we see that $|Q^{+}| > |Q_m|$.  Hence, if
  $Q^{+} \to C_m$ were an isometric embedding then this would
  contradict the choice of $u_m$ and $v_m$.  So, $Q^{+} \to C_m$ is
  not an isometric embedding and so $|Q^{+}| > \frac{|C_m|}{2}$.  But
  then
  \[ |Q^{+}| > \frac{|C_m|}{2} \ge \frac{|C|}{2} - \sum_{k < m}|Q_k|
    \ge \frac{|C|}{2} -
    (1-\xi)\biggl(\frac{K(3K-2)}{(K-1)^2}\biggr)\frac{|C|}{2} \] by
  \peqnref{sumbound} while
  \[ |Q_0| < (1-\xi)\biggl(\frac{K}{(K-1)}\biggr) \frac{|C|}{2} \] by
  \peqnref{onebound}.  So $|Q^{+}| \le |Q_0|$ would imply
  \[ 1 - (1-\xi)\biggl(\frac{K(3K-2)}{(K-1)^2}\biggr) <
    (1-\xi)\biggl(\frac{K}{(K-1)}\biggr) \] which after some
  manipulation gives $\xi < 1 - \frac{(K-1)^2}{K(4K-3)}$ which one can
  show contradicts our choice of
  $\xi > 1 - \frac{(K-1)^3}{13K^2(K+1)}$.  Hence $|Q^{+}| > |Q_0|$ so
  if $Q^{+} \to C$ were an isometric embedding then this would
  contradict the choice of $u_0$ and $v_0$.  So, $Q^{+} \to C$ is not
  an isometric embedding and so $|Q^{+}| > \frac{|C|}{2}$.  On the
  other hand
  \begin{align*}
    |Q^{+}| &\le |Q^{-}| + 2 \max_{i<j}|Q_i| \\
            &\le |Q_j^{-}|  + \sum_{i<j} \bigl(|Q_i|-|R_i|\bigr) +
              2 \max_{i<j}|Q_i| \\
            &\le \frac{|C_j|}{2} + \sum_{i<j} \bigl(|Q_i|-|R_i|\bigr) +
              2 \max_{i<j}|Q_i| \\
            &\le \frac{|C|}{2} + 2\sum_{i<j} \bigl(|Q_i|-|R_i|\bigr) +
              2 \max_{i<j}|Q_i| \\
            &\le \frac{|C|}{2} + 4\sum_{i<j} |Q_i| \\
            &\le \frac{|C|}{2} + (1-\xi)\biggl(\frac{4K(3K-2)}{(K-1)^2}\biggr)\frac{|C|}{2}
  \end{align*}
  where the last inequality follows from \peqnref{sumbound}.  We
  have
  \begin{equation}
    \eqnlabel{xibound}
    1 - \xi < \frac{(K-1)^3}{13K^2(K+1)} < \frac{(K-1)^2}{12K(K+1)} =
    \frac{(K-1)^2}{4K(3K+3)} < \frac{(K-1)^2}{4K(3K-2)}
  \end{equation}
  and so $|Q^{+}| < |C|$ so $Q^{+}$ embeds in $C$ and the endpoints
  $u,v$ of $Q^{+}$ in $C$ are at distance
  \[ d_C(u,v) \ge \frac{|C|}{2} -
    (1-\xi)\biggl(\frac{4K(3K-2)}{(K-1)^2}\biggr)\frac{|C|}{2}. \] But
  we also have
  \[ d_{\Gamma}\bigl(f(u),f(v)\bigr) \le |R^{+}| \le
    \frac{1}{K}|Q^{+}| \le \frac{1}{K}\Biggl(\frac{|C|}{2} +
    (1-\xi)\biggl(\frac{4K(3K-2)}{(K-1)^2}\biggr)\frac{|C|}{2}\Biggl) \]
  which, by \Propref{pqna}, implies
  \begin{align*}
    &\frac{1}{K}\Biggl(\frac{|C|}{2} +
      (1-\xi)\biggl(\frac{4K(3K-2)}{(K-1)^2}\biggr)\frac{|C|}{2}\Biggl) \\
    &\ge \frac{|C|}{2} - (1-\xi)\biggl(\frac{4K(3K-2)}{(K-1)^2}\biggr)\frac{|C|}{2} - (1-\xi)\frac{|C|}{2}
  \end{align*}
  which is equivalent to
  $1 - \xi \ge \frac{(K-1)^3}{K(13K^2 + 2K - 7)}$.  But
  \[ 1 - \xi < \frac{(K-1)^3}{13K^2(K+1)} =
    \frac{(K-1)^3}{K(13K^2+13K)} < \frac{(K-1)^3}{K(13K^2+2K - 7)} \]
  so we have a contradiction.  Therefore we have proved that
  $Q_i \subset \pc{0,i}$ for $1 < i < n$.

  Then the $Q_i$ are all pairwise disjoint in $C$ and
  $C \setminus \bigl(\bigcap_i\pc{i}\bigr) = \bigcup_iQ_i$ so
  $C_n$ is obtained from $C$ by replacing $Q_i \subset C$ with $R_i$,
  for each $i$.  Then since $f_n$ is $K$-bilipschitz and by
  \peqnref{sumbound} and \peqnref{xibound}, we have
  \begin{align*}
    \theta \ge |C_n| \ge |C| - \sum_i|Q_i|
    &\ge |C| - (1-\xi)\biggl(\frac{K(3K-2)}{(K-1)^2}\biggr)\frac{|C|}{2} \\
    & > |C| - \bigg(\frac{(K-1)^2}{4K(3K-2)}\biggr)\biggl(\frac{K(3K-2)}{(K-1)^2}\biggr)\frac{|C|}{2} \\
    &= \frac{7}{8}|C|
  \end{align*}
  So $|C| < \frac{8}{7}\theta$ and we see that $\frac{8}{7}\theta$
  bounds the lengths of $\xi$-almost isometric cycles of $\Gamma$.
\end{proof}

\begin{prop}
  \proplabel{subdiv} Let $\Gamma$ be a (strongly) shortcut graph.
  Then the graph obtained from $\Gamma$ by subdividing each edge is
  (strongly) shortcut.
\end{prop}
\begin{proof}
  Let $\Gamma'$ be the barycentric subdivision of $\Gamma$.  Then
  $\Gamma'$ is isometric to $\Gamma$ after scaling the metric by a
  factor of $2$.  Since isometric cycles are embedded, they have no
  backtracks and so every isometric cycle of $\Gamma'$ is the
  subdivision of an isometric cycle of $\Gamma$.  Hence, if $\theta$
  bounds the lengths of the isometric cycles of $\Gamma$ then
  $2\theta$ bounds the lengths of the isometric cycles of $\Gamma'$.
  So if $\Gamma$ is shortcut then $\Gamma'$ is shortcut.  Similary, if
  there is a bound on the lengths of the $K$-bilipschitz cycles of
  $\Gamma$ then there is a bound on the lengths of the $K$-bilipschitz
  cycles of $\Gamma'$.  So, by \Propref{bilip}, if $\Gamma$ is
  strongly shortcut then $\Gamma'$ is strongly shortcut.
\end{proof}

\begin{prop}
  \proplabel{free} Let $G$ be a (strongly) shortcut group.  Then $G$ acts
  freely and cocompactly on a (strongly) shortcut graph.
\end{prop}
\begin{proof}
  Let $G$ act properly and cocompactly on a (strongly) shortcut graph
  $\Gamma$.  If $\Gamma$ has a single vertex then $G$ is finite and so
  acts freely on any Cayley graph of $G$ which is strongly shortcut
  because it is connected and finite.  So we may assume that $\Gamma$
  has more than one vertex.
  
  By \Propref{subdiv}, we may assume that $G$ acts on $\Gamma$ without
  edge inversions.  Let $\pi \colon G \times \Gamma^0 \to \Gamma^0$ be
  the projection onto the second factor.  Define a graph
  $\tilde \Gamma$ on the vertex set $G \times \Gamma^0$ where
  $\tilde \Gamma$ has an edge joining $(g,v)$ and $(g',v')$ for each
  edge joining $v$ and $v'$.  Then the diagonal action
  $G \actson G \times \Gamma^0$ given by \[ g\cdot(g',v) = (gg',gv) \]
  extends to $\tilde \Gamma$ and the projection $\pi$ extends to a
  $G$-equivariant nondegenerate combinatorial map
  $\pi \colon \tilde \Gamma \to \Gamma$.  That $G$ acts on $\Gamma$
  without edge inversions rules out nontrivial fixed points of
  midpoints of edges of $\tilde \Gamma$ and so the action of $G$ on
  $\tilde \Gamma$ is free.  Let $\{v_1, v_2, \ldots, v_k\}$ be a set
  of orbit representatives of $G \actson \Gamma^0$.  Let $\hat \Gamma$
  be the induced subgraph of $\tilde \Gamma$ on
  $\bigcup_{g \in G}\bigl\{(g,gv_i)\bigr\}_i$.  We will prove that
  $\hat \Gamma$ is (strongly) shortcut and that the action of $G$ on
  $\hat \Gamma$ is cocompact.

  Let $\hat \pi\colon \hat \Gamma \to \Gamma$ be the restriction of
  $\pi$ to $\hat \Gamma$.  Since $G$ acts properly on $\Gamma$, the
  preimage under $\hat \pi$ of a vertex of $\Gamma$ is finite and so
  $\hat \Gamma$ is locally finite.  Also, the vertex set of
  $\hat \Gamma$ is the union of finitely many orbits
  $\bigcup_{i=1}^kG\cdot(1,v_i)$ so the action of $G$ on
  $\tilde \Gamma$ is cocompact.  It remains to prove that
  $\hat \Gamma$ is (strongly) shortcut.

  For a vertex $v \in \Gamma^0$, we have $v = gv_i$ for some $i$ and
  so $\hat\pi(g,gv_i) = v$.  Moreover, since $\hat \Gamma$ is an
  induced subgraph of $\tilde \Gamma$, for each pair of vertices
  $(g,u),(h,v) \in \hat\Gamma^0$, we see that $\hat \pi$ induces a
  bijection between the set of edges between $(g,u)$ and $(h,v)$ and
  the set of edges between $u$ and $v$.  This implies that for any
  $(g,u),(h,v) \in \hat\Gamma^0$, we can lift any path of nonzero
  length $\alpha \colon P \to \Gamma$ between $u$ and $v$ to a path
  $\hat \alpha \colon P \to \hat \Gamma$ from $(g,u)$ to $(h,v)$.  The
  lift is not unique since, if the sequence of vertices visited by
  $\alpha$ is $(u = u_0, u_1, u_2, \ldots, u_k = v)$ then, for
  $0 < i < k$, the lift of $u_i$ in $\hat \alpha$ may be any
  $(g,u_i) \in \hat \pi^{-1}(u_i)$.

  If $\Gamma$ is strongly shortcut then let $\theta \ge 3$ bound the
  lengths of the $\bar \xi$-almost isometric cycles of
  $\Gamma$. Otherwise, let $\theta \ge 3$ bound the lengths of the
  isometric cycles of $\Gamma$ and set $\bar \xi = 1$.  Let
  $\hat\xi = \frac{1+\bar\xi}{2}$ and let $f \colon C \to \hat \Gamma$
  be a ($\hat\xi$-almost) isometric a cycle of length $|C| > \theta$.
  By \Propref{aadist},
  \[ d_{\Gamma}\bigl(\hat\pi\comp f(u), \hat\pi \comp f(v)\bigr) <
    \bar \xi d_C(u,v) \] for some $u,v \in C^0$ with
  $d_C(u,v) \ge \Bigl\lfloor\frac{|C|}{2}\Bigr\rfloor - 1$.  If
  $d_{\Gamma}\bigl(\hat\pi\comp f(u), \hat\pi \comp f(v)\bigr) > 0$
  then let $\alpha \colon P \to \Gamma$ be a geodesic from
  $\hat\pi \comp f(u)$ to $\hat\pi \comp f(v)$.  Otherwise, let
  $\alpha \colon P \to \Gamma$ be a path of length $2$ from
  $\hat\pi \comp f(u)$ to $\hat\pi \comp f(v) = \hat\pi \comp f(u)$.
  This is always possible since $\Gamma$ is a connected graph on more
  than one vertex.  By the previous paragraph, we may lift $\alpha$ to
  a path $\hat \alpha \colon P \to \hat \Gamma$ from $f(u)$ to $f(v)$.
  So we see that either
  \[ d_{\hat \Gamma}\bigl(f(u), f(v)\bigr) < \bar \xi d_C(u,v) \] or
  \[ d_{\hat \Gamma}\bigl(f(u), f(v)\bigr) \le 2 \] and so, by
  \Propref{pqna}, one of
  \begin{equation}
    \eqnlabel{ineq}
    d_C(u,v) - (1 - \hat\xi)\frac{|C|}{2} < \bar\xi d_C(u,v)
  \end{equation}
  or
  \begin{equation}
    \eqnlabel{oineq}
    d_C(u,v) - (1- \hat\xi)\frac{|C|}{2} \le 2
  \end{equation}
  must hold.  Since $d_C(u,v) \ge \frac{|C|}{2} - \frac{3}{2}$ we see
  that \peqnref{oineq} gives the bound $|C| \le \frac{7}{\hat\xi}$.
  On the other hand, \peqnref{ineq} is equivalent to
  \[ (1-\bar\xi)d_C(u,v) < (1 - \hat\xi)\frac{|C|}{2} \] and so
  gives
  \[ (\hat\xi-\bar\xi)\frac{|C|}{2} < (1-\bar\xi)\frac{3}{2} \] which
  is impossible if $\bar\xi = 1$ and otherwise gives the bound
  $|C| \le \frac{3(1-\bar\xi)}{\hat\xi - \bar\xi}$.
\end{proof}

\section{Filling properties and disk diagrams}
\seclabel{fill}

In this section we study disk diagrams and the isoperimetric and
isodiametric functions of (strongly) shortcut graphs and groups.  Let
$\Gamma$ be a graph and let $\theta \in \N$.  For the purposes of the
current discussion, a cycle $C$ is always based and oriented.  Hence
two cycles $f_1,f_2 \colon C \to \Gamma$ may be distinct even if
$f_1 = f_2 \comp \psi$ for some $\psi \in \Aut(C)$.  Let
\[ S_{\theta} = \bigl\{f \colon C \to \Gamma \sth |C| \le \theta
  \bigr\} \] be the set of cycles in $\Gamma$ of length less than or
equal to $\theta$.  The \defterm{$\theta$-filling}
$F_{\theta}(\Gamma)$ is the $2$-complex whose $1$-skeleton is $\Gamma$
and whose $2$-skeleton has a unique $2$-cell with attaching map
$f \colon C \to \Gamma$ for each $f \in S_{\theta}$.  If a group $G$
acts on $\Gamma$ then $G$ acts on $S_{\theta}$ by
$g \cdot f = \phi_g \comp f$ where $\phi_g \in \Aut(\Gamma)$ is the
automorphism by which $g$ acts on $\Gamma$.  Thus the action of $G$ on
$\Gamma$ extends to an action on $F_{\theta}(\Gamma)$ such that an
element $g \in G$ stabilizes a $2$-cell $F$ if and only if $g$
stabilizes $F$ pointwise.

For $\theta,N \in \N$ with $3 \le \theta \le N$ and $\xi \in (0,1)$
consider the following property.
\begin{equation}
  \eqnlabel{A}
  \text{Every cycle $C \to \Gamma$ with $\theta < |C| \le N$ is not
    $\xi$-almost isometric.}
\end{equation}

\begin{rmk}
  \rmklabel{A} If $\Gamma$ is shortcut then $\Gamma$
  satisfies \peqnref{A} for $\theta$ bounding the lengths of isometric
  cycles of $\Gamma$, for any $N \ge \theta$ and for
  $\xi \in \bigl(\frac{N - 2}{N},1\bigr)$.  Of course, if $\Gamma$ is
  strongly shortcut, then it satisfies \peqnref{A} for a fixed $\xi$
  not depending on $N$.
\end{rmk}

\begin{cons}
  \conslabel{diskdiag} Let $\Gamma$ be a graph satisfying \peqnref{A}.
  Given a cycle $f \colon C \to \Gamma$ of length $|C| \le N$ we will
  inductively construct a disk diagram $D_f \to F_{\theta}(\Gamma)$
  for $f$.  If $|C| \le \theta$ then $D_f \to F_{\theta}(\Gamma)$ is
  just a single $2$-cell mapping to the $2$-cell of
  $F_{\theta}(\Gamma)$ whose attaching map is isomorphic to $f$.
  Otherwise $f$ is not $\xi$-almost isometric and so, by
  \Propref{aadist},
  \[ d_{\Gamma}\bigl(f(u),f(v)\bigr) < \xi d_C(u,v) \] for some pair
  of vertices $u,v \in C^0$ with
  $d_C(u,v) \ge \Bigl\lfloor\frac{|C|}{2}\Bigr\rfloor - 1$.  Let $P$
  and $Q$ be the two segments of $C$ joining $u$ and $v$.  Let
  $g \colon R \to \Gamma$ be a geodesic path from $f(u)$ to $f(v)$ in
  $\Gamma$ and note that $|R| < \xi\min\bigl\{|P|,|Q|\bigr\}$.  Glue
  $f$ and $g$ together along $u \sim g^{-1}\bigl(f(u)\bigr)$ and
  $v \sim g^{-1}\bigl(f(v)\bigr)$ to obtain a combinatorial map
  $h \colon (C \sqcup R)/{\sim} \to \Gamma$ with $h|_{P \cup R}$ and
  $h|_{Q \cup R}$ cycles of length \[ |P|+|R| < |P|+ \xi |Q| < |C| \] and
  \[ |Q|+|R| < |Q|+ \xi |P| < |C| \] and so, by induction we have disc
  diagrams $D_{h|_{P \cup R}} \to F_{\theta}(\Gamma)$ and
  $D_{h|_{Q \cup R}} \to F_{\theta}(\Gamma)$ for $h|_{P \cup R}$ and
  $h|_{Q \cup R}$.  Gluing $D_{h|_{P \cup R}} \to F_{\theta}(\Gamma)$
  and $D_{h|_{Q \cup R}} \to F_{\theta}(\Gamma)$ together along $R$ we
  obtain a disk diagram $D_f \to F_{\theta}(\Gamma)$ for $f$.
\end{cons}

\subsection{Simple connectedness}

\begin{thm}
  \thmlabel{ftsc} Let $\Gamma$ be a shortcut graph and let
  $\theta \ge 3$ bound the lengths of the isometric cycles of
  $\Gamma$.  Then $F_{\theta}(\Gamma)$ is simply connected.
\end{thm}
\begin{proof}
  Let $f \colon C \to \Gamma$ be a cycle in $\Gamma$.  Then, by
  \Rmkref{A} $\Gamma$ satisfies \peqnref{A} for $\theta$ as
  given, for $N=|C|$ and for $\xi = \frac{N - 1}{N}$.  Hence, we may
  apply \Consref{diskdiag} to obtain a disk diagram for $f$.
\end{proof}

\begin{cor}
  \corlabel{fundgroup} Let $G$ be a (strongly) shortcut group.  Then
  there is a compact $2$-complex $X$ with $G = \pi_1(X)$ such that the
  universal cover $\univcov{X}$ of $X$ has (strongly) shortcut
  $1$-skeleton.  In particular, the group $G$ is finitely presented.
\end{cor}
\begin{proof}
  By \Propref{free}, there is a free and cocompact action of $G$ on a
  (strongly) shortcut graph $\Gamma$.  In particular, the graph
  $\Gamma$ is shortcut so let $\theta \ge 3$ bound the lengths of the
  isometric cycles of $\Gamma$ and let
  \[ \mathscr{C} = \bigl\{f \colon (\vec{C},v) \to \Gamma \sth
    |\vec{C}| \le \theta \bigr\} \] be the set of all based oriented
  cycles of length at most $\theta$ in $\Gamma$.  Then $G$ acts on
  $\mathscr{C}$ so $G$ acts on the $2$-complex $\univcov{X}$ obtained
  from $\Gamma$ by gluing in a $2$-cell along each
  $f \in \mathscr{C}$.  The $2$-complex $\univcov{X}$ is a
  supercomplex of the $\theta$-filling $F_{\theta}(\Gamma)$ and
  $X^1 = F_{\theta}(\Gamma)^1 = \Gamma$ so, by \Thmref{ftsc}, the
  $2$-complex $\univcov{X}$ is simply connected.  The $G$-action on
  $\univcov{X}$ is free since if some $g \in G$ stabilizes a $2$-cell
  then it must fix its boundary so, by the freeness of the action on
  $\univcov{X}^1 = \Gamma$, we have $g = 1$.
\end{proof}

\subsection{Isoperimetric function}

\begin{thm}
  \thmlabel{isoper} Let $\Gamma$ be a graph.  If $\Gamma$ is shortcut
  then, for $\theta$ large enough, the Dehn function $\Delta$ of the
  filling $F_{\theta}(\Gamma)$ satisfies $\Delta(n) \le 2^n$.  If
  $\Gamma$ is strongly shortcut for $\xi \in (0,1)$ and
  $r > \frac{1}{1-\log_2(1+\xi)}$ then, for $\theta$ large enough, the
  Dehn function $\Delta$ of the filling $F_{\theta}(\Gamma)$ satisfies
  $\Delta(n) \le n^r$.
\end{thm}
\begin{proof}
  Suppose $\Gamma$ is shortcut and let $\theta \ge 3$ bound the
  lengths of the isometric cycles of $\Gamma$.  Let
  $\Delta \colon \N \to \N$ be the Dehn function of
  $F_{\theta}(\Gamma)$.  We will prove, by induction on $n$ that
  $\Delta(n) \le 2^n$.  If $n \le \theta$ then this clearly holds
  since any cycle of length at most $\theta$ bounds a $2$-cell in
  $F_{\theta}(\Gamma)$.  Let $f \colon C \to \Gamma$ be a cycle of
  length $n > \theta$.  Applying \Consref{diskdiag} to $f$ with $N=n$
  and $\xi = \frac{n-1}{n}$ we see that $f$ bounds a disk diagram
  $D_f$ which is the union of two disk diagrams of boundary length
  less than $n$.  Hence $f$ bounds a disk of area at most
  $2\Delta(n-1)$.  By induction
  \[ 2\Delta(n-1) \le 2\cdot 2^{n-1} = 2^n \] and so we have
  $\Delta(n) \le 2^n$.
  
  Suppose $\Gamma$ is strongly shortcut.  Choose $L \in \N$ with
  $L > 3$.  Let $\theta \ge \frac{L}{1-\xi}$ bound the lengths of the
  $\xi$-almost isometric cycles of $\Gamma$.  We will prove that the
  Dehn function of $F_{\theta}(\Gamma)$ satisfies
  $\Delta(n) \le n^{\log_b(2)}$ for $b = \frac{2L}{(L-3)\xi + L + 3}$.
  Note that $b > 1$ and that $b$ tends to $\frac{2}{1+\xi}$ as $L$
  goes to infinity so that $\log_{b}(2)$ tends to
  $\frac{\log_2(2)}{\log_2(\frac{2}{1+\xi})} = \frac{1}{1 -
    \log_2(1+\xi)}$.  So if $r > \frac{1}{1-\log_2(1+\xi)}$ we may
  choose $L$ large enough that $r > \log_b(2)$.  Going forward we
  assume that we have much such a choice of $L$.

  The argument proceeds as in the shortcut case but in the inductive
  step $f$ bounds a disk diagram which is the union of two disk
  diagrams of boundary length strictly less than
  \[ \xi\frac{n}{2} + \Bigl\lceil \frac{n}{2} \Bigr\rceil + 1 \le
    \xi\frac{n}{2} + \frac{n}{2} + \frac{1}{2} + 1 =
    \frac{1}{2}\Bigl(\xi + 1 + \frac{3}{n}\Bigr)n \le
    \frac{1}{2}\Bigl(\xi + 1 + \frac{3}{\theta}\Bigr)n \] so, since
  $\theta \ge \frac{L}{1-\xi}$ we have a disk diagram for $f$ of area
  at most
  \[ 2\Delta\Bigl(\Bigl\lfloor\frac{1}{2}\Bigl(\xi + 1 +
    \frac{3(1-\xi)}{L}\Bigr)n\Bigr\rfloor\Bigr) =
    2\Delta\Bigl(\Bigl\lfloor\frac{1}{2L}\Bigl((L-3)\xi +
    L+3\Bigr)n\Bigr\rfloor\Bigr) =
    2\Delta\Bigl(\Bigl\lfloor\frac{1}{b}n\Bigr\rfloor\Bigr)\] and so
  by induction we have
  \[ 2\Delta\Big(\Bigl\lfloor\frac{1}{b}n\Bigr\rfloor\Big) \le
    2\Big(\Bigl\lfloor\frac{1}{b}n\Bigr\rfloor\Big)^{\log_b(2)} \le
    2\Bigl(\frac{1}{b}n\Bigr)^{\log_b(2)} =
    2\Bigl(\frac{1}{2}n^{\log_b(2)}\Bigr) = n^{\log_b(2)} \] and so we
  have that $\Delta(n) \le n^{\log_b(2)} < n^r$.
\end{proof}

\begin{cor}
  \corlabel{isoper} Let $G$ be a group.  If $G$ is shortcut then it
  has an exponential isoperimetric function.  If $G$ is strongly
  shortcut then it has a polynomial isoperimetric function.
\end{cor}

\begin{cor}
  Let $G$ be a shortcut group.  Then $G$ has a decidable word problem.
\end{cor}

\begin{cor}
  Let $\Gamma$ be a graph that is strongly shortcut for some
  $\xi \in (0,\sqrt{2} - 1)$.  Then $\Gamma$ is hyperbolic.
\end{cor}
\begin{proof}
  We have
  $\frac{1}{1 - \log_2(1+\xi)} < \frac{1}{1-\log_2(\sqrt{2})} = 2$ so
  we can choose $r < 2$ such that $r > \frac{1}{1 - \log_2(1+\xi)}$.
  Then by \Thmref{isoper}, for $\theta$ large enough, the Dehn
  function $\Delta$ of the filling $F_{\theta}(\Gamma)$ satisfies
  $\Delta(n) \le n^r$ so is subquadratic.  Then, by the isoperimetric
  gap \cite{Gromov:1987, olshanski:subquadratic:1991,
    bowditch:subquadratic:1995}, $\Gamma$ has linear isoperimetric
  function and so is hyperbolic.
\end{proof}

\subsection{Isodiametric function}

\begin{thm}
  \thmlabel{isodiam} Let $\Gamma$ be a graph.  If $\Gamma$ is shortcut
  then, for $\theta$ large enough, the filling $F_{\theta}(\Gamma)$
  has an exponential isodiametric function.  If $\Gamma$ is strongly
  shortcut then, for $\theta$ large enough, the filling
  $F_{\theta}(\Gamma)$ has a polynomial isodiametric function.
\end{thm}
\begin{proof}
  For a cycle $f \colon C \to \Gamma$ let $\diam(f)$ denote the
  minimum diameter of a disk diagram for $f$.  Observe that in
  \Consref{diskdiag}
  $\diam(f) \le \diam(h|_{P\cup R}) + \diam(h|_{Q\cup R})$.  Indeed we
  may glue together minimal diameter disk diagrams of $h|_{P\cup R}$
  and $h|_{Q\cup R}$ along $R$ to obtain a disk diagram for $f$.
  Using this observation, the proof follows virtually identically to
  that of \Thmref{isoper}.
\end{proof}

\begin{cor}
  Let $G$ be a group.  If $G$ is shortcut then it has an exponential
  isodiametric function.  If $G$ is strongly shortcut then it has a
  polynomial isodiametric function.
\end{cor}

\section{Combinations}
\seclabel{combin}

In this section we show that (strongly) shortcut graphs and groups are
closed under products and that a finite graph of (strongly) shortcut
groups with finite edge groups is (strongly) shortcut.

\subsection{Products}

Let $\Gamma_1$ and $\Gamma_2$ be simplicial graphs.  The
\defterm{product graph} $\Gamma_1 \times \Gamma_2$ of $\Gamma_1$ and
$\Gamma_2$ is the $1$-skeleton of the CW complex product of $\Gamma_1$
and $\Gamma_2$.  The vertex set of $\Gamma_1 \times \Gamma_2$ is
$\Gamma_1^0 \times \Gamma_2^0$ and the edges of
$\Gamma_1 \times \Gamma_2$ are given by $(u_1,u_2) \sim (v_1,v_2)$
whenever
\[ \text{$u_1 = v_1$ and $u_2 \sim v_2$} \] or
\[ \text{$u_1 \sim v_1$ and $u_2 = v_2$} \] where $\sim$ is the edge
relation.

\begin{thm}
  \thmlabel{prod} Let $\Gamma_1$ and $\Gamma_2$ be (strongly) shortcut
  graphs.  Then $\Gamma_1 \times \Gamma_2$ is (strongly) shortcut.
\end{thm}
\begin{proof}
  Let $\Gamma_1$ and $\Gamma_2$ be (strongly) shortcut and let
  $\theta$ bound the lengths of the ($\xi$-almost) isometric cycles of
  the $\Gamma_i$.  Let $\Gamma = \Gamma_1 \times \Gamma_2$ and let
  $f \colon C \to \Gamma$ be a cycle of length $|C| \ge 2\theta$.  We
  combine the shortcut and strongly shortcut cases as follows.  If the
  $\Gamma_i$ are strongly shortcut then we have $\theta$ and $\xi$ as
  given.  Otherwise, by \Rmkref{A}, the $\Gamma_i$ satisfy \peqnref{A}
  for $\theta$ as given, for $N = |C|$ and for some $\xi$ depending on
  $N$.  We will show that for some antipodal pair of points
  $p,q \in C$ we have
  $d_{\Gamma}\bigl(f(p),f(q)\bigr) < \bigl(\frac{1 +
    \xi}{2}\bigr)\frac{|C|}{2}$.

  Each edge of $C$ projects nondegenerately onto exactly one of
  $\Gamma_1$ or $\Gamma_2$.  Call those edges that project
  nondegenerately onto $\Gamma_1$ \defterm{horizontal edges} and those
  that project nondegenerately onto $\Gamma_2$ \defterm{vertical
    edges}.  Without loss of generality the number of horizontal edges
  is greater than or equal to the number of vertical edges.  Let
  $f_1 \colon C_1 \to \Gamma_1$ be the cycle obtained from
  $C \to \Gamma_1$ by contracting the vertical edges of $C$.  Then
  $N \ge |C_1| \ge \frac{|C|}{2} \ge \theta$ so we have
  $d_{\Gamma_1}\bigl(f_1(p_1),f_1(q_1)\bigr) < \xi d_{C_1}(p_1,q_1)$
  for some antipodal pair of points $p_1,q_1 \in C_1$.  Let
  $p',q' \in C$ map to $p_1$ and $q_1$ under the contraction map
  $C \to C_1$.  We may choose $p'$ and $q'$ so that they are not
  contained in the interior of any vertical edge.  Let $\ell$ be the
  number of vertical edges in a geodesic segment of $C$ between $p'$
  and $q'$.  Then $d_C(p',q') = \frac{|C_1|}{2} + \ell$ while
  $d_{\Gamma}\bigl(f(p'),f(q')\bigr) < \xi \frac{|C_1|}{2} + \ell$.
  Let $P \subset C$ be a geodesic segment of length $\frac{|C|}{2}$
  containing $p'$ and $q'$ and let $p$ and $q$ be the endpoints of $P$
  with $p$ nearest to $p'$ and $q$ nearest to $q'$.  Then
  $d_C(p,q) = \frac{|C|}{2}$ while
  \begin{align*}
    d_{\Gamma}\bigl(f(p),f(q)\bigr) & \le
    d_{\Gamma}\bigl(f(p),f(p')\bigr) +
    d_{\Gamma}\bigl(f(p'),f(q')\bigr) +
    d_{\Gamma}\bigl(f(q'),f(q)\bigr) \\
    & \le d_C(p,p') +
    d_{\Gamma}\bigl(f(p'),f(q')\bigr) +
    d_C(q',q) \\
    &= d_C(p,q) - d_C(p',q') +
      d_{\Gamma}\bigl(f(p'),f(q')\bigr) \\
    &< \frac{|C|}{2} -
    \Bigl(\frac{|C_1|}{2} + \ell\Bigr) + \xi \frac{|C_1|}{2} + \ell \\
    &= \frac{|C|}{2} - (1 - \xi) \frac{|C_1|}{2} \\
    &= \Bigl(\frac{1 + \xi}{2}\Bigr)\frac{|C|}{2}
  \end{align*}
  and so
  $d_{\Gamma}\bigl(f(p),f(q)\bigr) < \bigl(\frac{1 +
    \xi}{2}\bigr)\frac{|C|}{2}$.
\end{proof}

\begin{cor}
  Let $G_1$ and $G_2$ be (strongly) shortcut groups.  Then
  $G_1 \times G_2$ is (strongly) shortcut.
\end{cor}

\subsection{Trees of shortcut graphs}

Let $T$ be a tree.  An \defterm{arc decomposition} of $T$ is a
partition of the set of edges of $T$ such that the following
conditions hold.
\begin{enumerate}
\item The union of the edges in each part is isomorphic to a path,
  which we refer to as an \defterm{arc}.
\item The interior vertices of each arc all have degree two.
\end{enumerate}
The endpoints of the arcs of an arc decomposition are called
\defterm{nodes}.  Every tree comes equipped with a default arc
decomposition whose arcs are simply its edges.  A \defterm{tree of
  graphs with discrete edge graphs} is a surjective (possibly
degenerate) combinatorial map $\Gamma \to T$ from a graph $\Gamma$ to
a tree $T$ that is equipped with an arc decomposition such that the
following conditions hold.
\begin{enumerate}
\item The preimage of each node $v$ is a connected subgraph $\Gamma_v$
  called the \defterm{vertex graph} at $v$ of $\Gamma \to T$.
\item If $P$ is an arc of $T$ with endpoints $u$ and $v$ then the
  preimage of $P \setminus \{u,v\}$ is a disjoint union
  $\bigsqcup_{\alpha} \bigl(\tilde P_{\alpha} \setminus \{\tilde
  u_{\alpha},\tilde v_{\alpha}\}\bigr)$ where, for each $\alpha$, the
  subgraph $\tilde P_{\alpha}$ maps isomorphically onto $P$ with
  $\tilde u_{\alpha}$ and $\tilde v_{\alpha}$ mapping onto $u$ and
  $v$.  The paths $\tilde P_{\alpha}$ are called \defterm{lifts} of
  $P$.
\end{enumerate}
The preimage of the midpoint of an arc $P$ of $T$ is a discrete set
called the \defterm{edge graph} $\Gamma_P$ at $P$ of $\Gamma \to T$.
For each endnode $v$ of an arc $P$ of $T$ there is a function
$\Gamma_{v,P} \colon \Gamma_P \to \Gamma_v$ sending $p \in \Gamma_P$
to the unique vertex contained in $\tilde P \cap \Gamma_v$ where
$\tilde P$ is the lift of $P$ that contains $p$.  The $\Gamma_{v,P}$
are called the \defterm{attaching maps} of $\Gamma \to T$.

\begin{rmk}
  The graph $\Gamma$ has the structure of a graph of spaces where the
  underlying graph is a tree, the vertex spaces are graphs, the edge
  spaces are discrete graphs and the attaching maps are combinatorial
  maps.  The theory of graphs of spaces is developed in Scott and Wall
  \cite{Scott:1979}.  Here we subdivide the edges of the tree into
  longer arcs so that the vertex spaces are at greater distance to
  each other in $\Gamma$.  We will make use of this in the proof
  \Thmref{toig}.
\end{rmk}

For the purpose of discussing trees of graphs, it is convenient to
consider combinatorial maps $C \to \Gamma$, with $C$ homeomorphic to
$S^1$, which are not necessarily nondegenerate.  Call such a map a
\defterm{cycle with degeneracies}.

The following lemma is a variant of Jordan's separator theorem for
trees \cite{Jordan:1869}.

\begin{lem}
  \lemlabel{treecycle} Let $T$ be a tree.  Let $f \colon C \to T$ be a
  cycle with degeneracies of length $|C| \ge 3$.  Then for some vertex
  $w \in f(C^0)$, the metric subspace $f^{-1}(w)$ has diameter at
  least $\frac{|C|}{3}$.
\end{lem}
\begin{proof}
  For $w$ in the image of $f$, consider the metric subspace
  $f^{-1}(w) \subset C$.  Choose $w$ in the image of $f$ such that
  $f^{-1}(w)$ has the largest possible diameter.  Let the vertices
  $u,v \in f^{-1}(w)$ achieve the diameter of $f^{-1}(w)$.  Suppose,
  for the sake of finding a contradiction, that
  $d_C(u,v) < \frac{|C|}{3}$.  Then the segment $P$ of length
  $|P| \ge \frac{2|C|}{3}$ between $u$ and $v$ in $C$ intersects
  $f^{-1}(w)$ only at $u$ and $v$.  Indeed any point $p \in f^{-1}(w)$
  must be at distance less than $\frac{|C|}{3}$ to both $u$ and $v$
  and there is no such point of $P$.  But then the first and last
  edges of $P$ map nondegenerately to some common edge $ww'$ of $T$.
  But then $f^{-1}(w')$ has diameter larger than the diameter of
  $f^{-1}(w)$ contradicting our choice of $w$.  Hence
  $d_C(u,v) \ge \frac{|C|}{3}$.
\end{proof}

\begin{lem}
  \lemlabel{distrib} Let $T$ be a tree.  Let $f \colon C \to T$ be a
  cycle with degeneracies of length $|C| \ge 3$.  Let
  $L < \frac{|C|}{3} + 1$ and suppose that for each edge
  $e \subset f(C)$, the distance between the midpoints of any two
  edges of $f^{-1}(e)$ is at most $L$.  Then for some vertex
  $w \in T^0$, any segment $P \subset C$ whose interior is disjoint
  from $f^{-1}(w)$ has length $|P| \le L + 1$.
\end{lem}
\begin{proof}
  Let $w$ be as in \Lemref{treecycle} and let $P$ be the closure of a
  component of $C \setminus f^{-1}(w)$.  We need to show that
  $|P| \le L + 1$.  The initial and terminal edges $e_1$ and $e_2$ of
  $P$ map to the same edge of $T$ and so either $|P| \le L + 1$ or
  $|P| \ge |C| - L + 1$.  But $|C| - L + 1 > \frac{2|C|}{3}$ while, by
  our choice of $w$, we have $|P| \le \frac{2|C|}{3}$.  Hence
  $|P| \le L + 1$ as required.
\end{proof}

A cycle with degeneracies is \defterm{$\xi$-almost isometric} if
\[ d_{\Gamma}\bigl(f(p),f(q)\bigr) \ge \xi \frac{|C|}{2} \] for any
antipodal pair of points $p,q \in C$.

\begin{lem}
  \lemlabel{degip} Let $\Gamma$ be strongly shortcut with $\theta$
  bounding the lengths of the $\xi$-almost isometric cycles of
  $\Gamma$.  Then there exist $\xi' \in (0,1)$ and $\theta' \in \N$
  depending only on $\xi$ and $\theta$ such that $\theta'$ bounds the
  lengths of the $\xi'$-almost isometric cycles with degeneracies of
  $\Gamma$.
\end{lem}
\begin{proof}
  Let $\theta$ bound the lengths of the $\xi$-almost isometrically
  embedded cycles of $\Gamma$.  Let $f \colon C \to \Gamma$ be a cycle
  with degeneracies.  Define $S \subset C$ as the union of all edges
  of $C$ that map to vertices under $f$.  We may assume that
  $S \neq C$ since, otherwise $f$ is the constant map and so satisfies
  the conclusion of the lemma for any $\xi'$.  Let
  $(P_i)_{i=1}^{\ell}$ be the sequence of components of $S$ in the
  order they are visited in some traversal of $C$.  We begin by
  showing, for $\theta'' \ge \max\bigl\{\theta,\frac{8}{1-\xi}\bigr\}$
  and $\xi'' = \frac{1 + \xi}{2}$, that if $|C| > \theta''$ and
  $|P_i|$ is even for each $i$ then there exist antipodal $p,q \in C$
  such that $d_{\Gamma}\bigl(f(p),f(q)\bigr) < \xi''\frac{|C|}{2}$.
  Call the condition that the $|P_i|$ are even the \defterm{parity
    condition}.  Later we will use this result to prove the statement
  of the lemma, for
  $\theta' \ge \max\bigl\{\theta'', \frac{2(2+\xi'')}{1-\xi''}\bigr\}$
  and $\xi' = \frac{1 + \xi''}{2}$, with no assumption on the parities
  of the $|P_i|$.

  Assume $f$ satisfies the parity condition and $|C| > \theta''$.  We
  obtain from $f$ a cycle without degeneracies
  $f' \colon C \to \Gamma$ by setting
  $f'|_{C \setminus S} = f|_{C \setminus S}$ and mapping each edge of
  $P_i$ onto $f(e_i)$ where $e_i$ is the edge that follows $P_i$ in
  some fixed orientation of $C$.  This is possible since $f$ satisfies
  the parity condition.  Thus $f'$ folds $P_i$ onto $f(e_i)$ in a
  zig-zag fashion.  Then for any point $p \in C$, we have
  $d_{\Gamma}\bigl(f(p),f'(p)\bigr) \le 1$.  If $|C| > \theta''$ then
  there is a pair of antipodal points $p,q \in C$ such that
  $d_{\Gamma}\bigl(f'(p),f'(q)\bigr) \le \xi \frac{|C|}{2}$.  Hence
  \[ d_{\Gamma}\bigl(f(p),f(q)\bigr) < \xi \frac{|C|}{2} + 2 =
    \Bigl(\xi + \frac{4}{|C|}\Bigr)\frac{|C|}{2} \le \Bigl(\xi +
    \frac{4}{\theta''}\Bigr)\frac{|C|}{2} \le \xi''\frac{|C|}{2} \] as
  required.

  We now consider a general $f\colon C \to \Gamma$ that does not
  necessarily satisfy the parity condition.  Assume that
  $|C| > \theta'$.  Let $i_0 < i_1 < \cdots < i_{m-1}$ be the set of
  indices for which $|P_i|$ is not even and assume $|C| > \theta'$.
  Obtain a cycle $f' \colon C' \to \Gamma$ from $f$ by contracting an
  edge in each $P_{i_j}$ with $j$ odd and expanding a vertex $v$ to an
  edge $e \mapsto f(v)$ in each $P_{i_j}$ with $j$ even.  Then $f'$
  satisfies the parity condition and we have
  $|C| \le |C'| \le |C| + 1$.  There is a relation
  $R \subset C \times C'$ with $pRq$ if and only if one of the
  following holds.
  \begin{enumerate}
  \item $p$ was obtained directly from $q$
  \item $p$ is contained in an edge that was contracted to $q$
  \item $p$ is a vertex which was expanded to an edge that contains
    $q$
  \end{enumerate}
  By the alternating nature of the expansions and contractions we see
  that if $pRp'$ and $qRq'$ then $d_C(p,q) \ge d_{C'}(p',q') - 1$.  By
  the previous paragraph, we have a pair of antipodal points
  $p',q' \in C'$ such that
  $d_{\Gamma}\bigl(f'(p'),f'(q')\bigr) < \xi''\frac{|C'|}{2}$.  Take
  any $p'',q'' \in C$ satisfying $p''Rp'$ and $q''Rq'$.  Then
  $f(p'') = f'(p')$ and $f(q'') = f'(q')$ and so
  $d_{\Gamma}\bigl(f(p''),f(q'')\bigr) < \xi''\frac{|C'|}{2} \le
  \xi''\frac{|C|+1}{2}$ and yet $d_C(p'',q'') \ge \frac{|C|}{2} - 1$.
  Hence, as $f$ is $1$-Lipschitz, for some antipodal pair of points
  $p,q \in C$ we have
  \[ d_{\Gamma}\bigl(f(p),f(q)\bigr) < \xi''\frac{|C|+1}{2} + 1 =
    \Bigl(\xi'' + \frac{\xi''+2}{|C|}\Bigr)\frac{|C|}{2} < \Bigl(\xi''
    + \frac{\xi''+2}{\theta'}\Bigr)\frac{|C|}{2} \le
    \xi'\frac{|C|}{2} \] as required.
\end{proof}

\begin{thm}
  \thmlabel{toig} Let $\phi \colon \Gamma \to T$ be a tree of graphs
  with discrete edge graphs satisfying the following conditions.
  \begin{enumerate}
  \item The vertex graphs $\Gamma_v$ are uniformly (strongly) shortcut
    in the sense that there exists $\theta \ge 3$ (and
    $\xi \in (0,1)$) such that $\theta$ bounds the lengths of the
    ($\xi$-almost) isometric cycles of every vertex graph.
  \item The arcs of $T$ all have some common length $M$ such that, for
    every attaching map $\Gamma_{v,P}$ of $\phi$, the diameter of
    $\Gamma_{v,P}(\Gamma_P)$ is at most $M$.
  \end{enumerate}
  Then $\Gamma$ is (strongly) shortcut.
\end{thm}
\begin{proof}
  We will first consider the case where the vertex graphs are
  shortcut.  Let $f \colon C \to \Gamma$ be an isometric cycle.  If
  $f$ maps entirely into a single vertex graph then $|C| \le \theta$
  by hypothesis.  So, suppose the image of $f$ contains some edge in
  the lift $\tilde P$ of an arc $P$ of $T$.  Then, since $f$ is
  injective, it must traverse all of $\tilde P$ and, by consideration
  of $\phi \comp f$, it must also traverse some other lift $\hat P$ of
  $P$ in the opposite direction.  Let $Q$ and $Q'$ be the segments of
  $C$ which map isomorphically to $\tilde P$ and $\hat P$.  Let
  $u \in Q$ and $u' \in Q'$ be endpoints of $Q$ and $Q'$ mapping to
  the same vertex graph.  Then
  $d_{\Gamma}\bigl(f(u),f(u')\bigr) \le M$ and so $d_C(u,u') \le M$
  whereas $Q$ and $Q'$ each have length $M$.  Hence the geodesic
  segment $R$ of $C$ between $u$ and $u'$ is disjoint from the
  interiors of $Q$ and $Q'$.  The same goes for the geodesic segment
  $R'$ between the other pair of endpoints of $Q$ and $Q'$.  Then $C$
  is covered by the segments $Q$, $Q'$, $R$ and $R'$, each of which
  has length at most $M$.  Hence the lengths of the isometric cycles
  of $\Gamma$ are bounded by $\max\{\theta, 4M\}$.
  
  The case where the vertex graphs are strongly shortcut requires a
  more delicate argument relying on the preceeding lemmas.  By
  \Lemref{degip} we can replace $\theta$ and $\xi$ so that $\theta$
  bounds the lengths of the $\xi$-almost isometric cycles with
  degeneracies of all the vertex graphs.  Let $\xi' = \frac{\xi+2}{3}$
  and let $\theta' = \max\bigl\{\theta,\frac{24M+6}{1-\xi}\bigr\}$.
  Let $f \colon C \to \Gamma$ be a $\xi'$-almost isometric cycle.  We
  will prove that $|C| \le \theta'$.  Since
  $\frac{24M+6}{1-\xi} \ge 4$ we may assume that $|C| \ge 4$.
  
  If the image of $f$ is contained entirely in a single vertex graph
  then $|C| \le \theta \le \theta'$ since $\xi' \ge \xi$.  So let us
  assume that $f$ is not confined to a single vertex graph.  Then
  $\phi \comp f$ maps some pair of distinct edges $e$ and $e'$ of $C$
  onto a common edge of some arc $P$ of $T$.  Then $f$ maps $e$ and
  $e'$ onto a pair of edges in the same relative position in lifts of
  $P$.  Let $\bar e$ and $\bar e'$ be the images of $e$ and $e'$ under
  $f$.  Since $P$ has length $M$ and the attaching maps of $\phi$ have
  diameter bounded by $M$, we have
  $d_{\Gamma}(p_{\bar e},p_{\bar e'}) \le 2M$ where $p_{\bar e}$ and
  $p_{\bar e'}$ are the midpoints of $\bar e$ and $\bar e'$. Then by
  \Propref{pqna} $d_C(p_e,p_{e'}) \le 2M + (1 - \xi')\frac{|C|}{2}$
  where $p_e$ and $p_{e'}$ are the midpoints of $e$ and $e'$.  Let
  $L = 2M + (1 - \xi')\frac{|C|}{2}$.  If $L \ge \frac{|C|}{3} + 1$
  then we have
  \begin{align*}
    0 &\le 2M + \Bigl(1 - \frac{\xi + 2}{3}\Bigr)\frac{|C|}{2} - \frac{|C|}{3} - 1 \\
      &= 2M - 1 + \Bigl(\frac{1}{2} - \frac{\xi + 2}{6} - \frac{1}{3}\Bigl)|C| \\
      &= 2M - 1 - \Bigl(\frac{\xi + 1}{6}\Bigr)|C|
  \end{align*}
  and so
  $|C| \le \frac{6(2M-1)}{\xi+1} \le \frac{24M+6}{1-\xi} \le \theta'$.
  So we may assume that $L < \frac{|C|}{3} + 1$ and can apply
  \Lemref{distrib} to $\phi \comp f$ and $L$ to obtain a vertex
  $w \in T^0$ such that any segment $Q \subset C$ whose interior is
  disjoint from $(\phi \comp f)^{-1}(w)$ has length $|Q| \le L+1$.
  Let $v \in (\phi \comp f)^{-1}(w)$ be a vertex.

  Suppose $w$ is an interior vertex of an arc $P$ of $T$.  Let $p$ be
  the antipode of $v$ and let $p'$ be a point of
  $(\phi \comp f)^{-1}(w)$ that is nearest to $p$.  Then
  $d_C(p,p') \le \frac{L+1}{2}$ and so
  $d_C(p',v) \ge \frac{|C|}{2} - \frac{L+1}{2}$.  So, since arcs have
  length $M$ and the images of attaching maps of $\phi$ have diameter
  at most $M$, we have
  \begin{align*}
    2M &\ge d_{\Gamma}\bigl(f(p'),f(v)\bigr) \\
       & \ge d_C(p',v) - (1-\xi')\frac{|C|}{2} \\
       &\ge \frac{|C|}{2} - \frac{L+1}{2} - (1-\xi')\frac{|C|}{2} \\
       &= \xi'\frac{|C|}{2} - \frac{L+1}{2}
  \end{align*}
  where the second inequality holds by \Propref{pqna}.  So, recalling
  that $L = 2M + (1 - \xi')\frac{|C|}{2}$ we have
  \[ 2M \ge \xi'\frac{|C|}{2} - M - (1 - \xi')\frac{|C|}{4} -
    \frac{1}{2} \] which gives
  $|C| \le \frac{4M+2}{3\xi' - 1} = \frac{4M+2}{\xi+1} \le
  \frac{24M+6}{1-\xi} \le \theta'$.
  
  Suppose $w$ is a node of $T$.  Then
  $(\phi \comp f)^{-1}(w) = f^{-1}(\Gamma_w)$.  Let $(P_i)_i$ be the
  components of $f^{-1}(\Gamma_w)$ and let $(Q_j)_j$ be the closures
  of the components of $C \setminus f^{-1}(\Gamma_w)$.  Then
  $|Q_j| \le L+1$ for each $j$ and $f$ maps each $P_i$ into $\Gamma_w$
  and maps each $Q_j$ into the closure of the complement of
  $\Gamma_w$.  We will define a cycle with degeneracies
  $f' \colon C \to \Gamma_w$ that agrees with $f$ on the $P_i$ and
  that maps each $Q_j$ onto a geodesic of $\Gamma_w$.  To see that
  this is possible, we need only to show that the endpoints of each
  $Q_j$ map to a distance of at most $|Q_j|$ in $\Gamma_w$.  The
  endpoints of $Q_j$ map to a distance of at most $M$ since $M$ bounds
  the diameters of the attaching maps of $\phi$.  So we need only
  consider the case where $|Q_j| < M$.  But then $Q_j$ is not long
  enough for $f|_{Q_j}$ to traverse the lift of an arc of $T$ since
  the arcs have length $M$.  Hence the endpoints of $Q_j$ map to the
  same vertex of $\Gamma_w$.  So we are able to define
  $f' \colon C \to \Gamma_w$.  Then, for a point $p \in C$, we have
  $d_{\Gamma}\bigl(f(p),f'(p)\bigr) \le M + \frac{L+1}{2}$.
  So if $p,q\in C$ are any antipodal pair then
  \begin{align*}
    d_{\Gamma_w}\bigl(f'(p),f'(q)\bigr)
    & \ge d_{\Gamma}\bigl(f'(p),f'(q)\bigr) \\
    & \ge d_{\Gamma}\bigl(f(p),f(q)\bigr) - 2M - L - 1 \\
    & \ge \xi'\frac{|C|}{2} - 2M - L - 1
  \end{align*}
  and so
  $d_{\Gamma_w}\bigl(f'(p),f'(q)\bigr) \ge \bigl(\xi' -
  \frac{2(2M+L+1)}{|C|}\bigr)\frac{|C|}{2}$.  So as long as
  $\xi' - \frac{2(2M+L+1)}{|C|} \ge \xi$ then $f'$ is $\xi$-almost
  isometric and thus we have $|C| \le \theta \le \theta'$.  If
  $\xi' - \frac{2(2M+L+1)}{|C|} < \xi$ then we have
  \begin{align*}
    0 &< \xi - \frac{\xi + 2}{3} + \frac{2\bigl(2M+2M + \bigl(1 - \frac{\xi + 2}{3}\bigr)\frac{|C|}{2} + 1\bigr)}{|C|} \\
      &= \frac{2\xi - 2}{3} + \frac{8M}{|C|} + \Bigl(1 - \frac{\xi + 2}{3}\Bigr) + \frac{2}{|C|} \\
    &= \frac{\xi - 1}{3} + \frac{8M+2}{|C|}
  \end{align*}
  and so $|C| < \frac{24M+6}{1-\xi} \le \theta'$.
\end{proof}

\begin{cor}
  \corlabel{gosg} Let $\mathscr{G}$ be a finite graph of (strongly)
  shortcut groups with finite edge groups.  Then the fundamental group
  of $\mathscr{G}$ is (strongly) shortcut.
\end{cor}
\begin{proof}
  Let $\Gamma$ be the underlying graph of $\mathscr{G}$.  We construct
  a graph of spaces $\mathscr{H}$ on $\Gamma$ such that the
  fundamental group functor sends $\mathscr{H}$ to $\mathscr{G}$.  See
  Scott and Wall for this viewpoint on graphs of groups
  \cite{Scott:1979}.  By \Corref{fundgroup}, we can choose the vertex
  spaces so that their universal covers have (strongly) shortcut
  $1$-skeleton.  The $1$-skeleton $\tilde \Gamma$ of the universal
  cover of $\mathscr{H}$ has the structure $\tilde \Gamma \to T$ of a
  tree of graphs $\tilde \Gamma \to T$ where $T$ is the Bass-Serre
  tree of $\mathscr{G}$.  The fundamental group $\pi_1(\mathscr{G})$
  acts freely and cocompactly on $\Gamma$.  For $M$ large enough,
  subdividing each edge of $T$ into an arc of length $M$ results in a
  tree of graphs that satisfies the conditions of \Thmref{toig}.
\end{proof}

\begin{cor}
  Amalgamations and HNN extensions of (strongly) shortcut groups
  over finite subgroups are (strongly) shortcut.
\end{cor}

Note that $\BS(1,2)$ is an HNN extension of $\Z$ but is not strongly
shortcut.  Hence, we see that the condition that the edge groups be
finite is essential in the strong shortcut case.

\section{Examples}
\seclabel{examp}

In this section we prove that hyperbolic graphs, $1$-skeletons of
$\CAT(0)$ cube complexes, the standard Cayley graphs of finitely
generated Coxeter groups and all Cayley graphs of $\Z$ and $\Z^2$ are
strongly shortcut.  In particular, hyperbolic groups, cocompactly
cubulated groups and finitely generated Coxeter groups are strongly
shortcut.

\subsection{Hyperbolic graphs}

The thinness of geodesic bigons in a hyperbolic graph immediately
implies the shortcut property.  In this section we will prove that
hyperbolic graphs are in fact strongly shortcut.  To do so we will
make use of the following proposition whose proof is given in Bridson
and Haefliger \cite{Bridson:1999}.

\begin{prop}[Specialization of Proposition~1.6 of Part~III of
  \proplabel{bridson} Bridson and Haefliger \cite{Bridson:1999}]
  Let $\Gamma$ be a $\delta$-hyperbolic graph.  Let
  $f \colon P \to \Gamma$ be a $1$-Lipchitz map to $\Gamma$ from a
  compact interval $P \subset \R$.  If $Q \subset \Gamma$ is the image
  of a geodesic joining the endpoints of $f$, then
  \[ d_{\Gamma}\bigl(x,f(P)\bigr) \le \delta \max\{0,\log_2|P|\} +
    1 \] for every $x \in Q$.
\end{prop}

\begin{thm}
  \thmlabel{hypgraph} Let $\Gamma$ be a hyperbolic graph.  Then $\Gamma$ is
  strongly shortcut.
\end{thm}
\begin{proof}
  Let $\delta \ge 1$ be a hyperbolicity constant for $\Gamma$.
  Suppose $f \colon C \to \Gamma$ is a $\frac{3}{4}$-almost isometric
  cycle of length $|C| \ge 2$.  Let $y,y' \in C$ be a pair of
  antipodal points and let $P_1 \subset C$ and $P_2 \subset C$ be the
  two segments of $C$ between $y$ and $y'$.  Let $Q$ be the image of a
  geodesic in $\Gamma$ from $f(y)$ to $f(y')$ and let $x$ be the
  midpoint of $Q$.  Then, by \Propref{bridson}, there are points
  $p_1 \in P_1$ and $p_2 \in P_2$ such that $f(p_1)$ and $f(p_2)$ are
  each at distance at most $\delta \log_2\frac{|C|}{2} + 1$ from $x$
  in $\Gamma$.  Then, since $f$ is $\frac{3}{4}$-almost isometric, we
  have $|Q| \ge \frac{3|C|}{8}$ and
  \begin{align*}
    \frac{3|C|}{16} \le \frac{1}{2}|Q| &= d_{\Gamma}\bigl(f(y),x\bigr) \\
    &\le d_{\Gamma}\bigl(f(y),f(p_1)\bigr) +
    d_{\Gamma}\bigl(f(p_1),x\bigr) \\
    & \le d_C(y,p_1) + \delta \log_2\frac{|C|}{2} + 1
  \end{align*}
  and so
  $d_C(y,p_1) \ge \frac{3|C|}{16} - \delta \log_2\frac{|C|}{2} - 1$.
  By the same argument we have the same lower bound for $d_C(y,p_2)$
  and $d_C(y',p_1)$ and $d_C(y',p_2)$.  Hence
  \[ d_C(p_1,p_2) \ge \frac{3|C|}{8} - 2 \delta \log_2\frac{|C|}{2} -
    2 \] and so, by \Propref{pqna}, we have
  \[ d_{\Gamma}\bigl(f(p_1),f(p_2)\bigr) \ge \frac{3|C|}{8} - 2 \delta
    \log_2\frac{|C|}{2} - 2 - \Bigl(1 - \frac{3}{4}\Bigr)\frac{|C|}{2}
    = \frac{|C|}{4} - 2 \delta \log_2\frac{|C|}{2} - 2 \] but $f(p_1)$
  and $f(p_2)$ are both within a distance of
  $\delta \log_2\frac{|C|}{2} + 1$ to $x$ and so
  $d_{\Gamma}\bigl(f(p_1),f(p_2)\bigr) \le 2\delta \log_2\frac{|C|}{2}
  + 2$.  Hence we have
  \[ |C| \le 16\Bigl(\delta\log_2\frac{|C|}{2} + 1\Bigr) \] which
  bounds the length $|C|$ of $f$.
\end{proof}

\begin{cor}
  Hyperbolic groups are strongly shortcut.
\end{cor}

\subsection{\texorpdfstring{$\CAT(0)$}{CAT(0)} cube complexes}

In this section we will prove that the $1$-skeleton of a
finite-dimensional $\CAT(0)$ cube complex is strongly shortcut.  The
proof rests on a theorem about edge colorings of cycles.

Let $C$ be a cycle.  An \defterm{edge coloring} of $C$ is a function
$\alpha \colon C^{(1)} \to W$ from the set $C^{(1)}$ of edges of $C$
to some set $W$ of \defterm{colors}.  A cycle $C$ along with an edge
coloring $\alpha \colon C^{(1)} \to W$ is a \defterm{wall cycle} if
$\alpha$ is surjective and, for each $w \in W$, the number
$|\alpha^{-1}(w)|$ of edges of color $w$ is even.  In this case we may
refer to the elements of $W$ as \defterm{walls}.

Let $(C, \alpha)$ be a wall cycle.  A combinatorial segment
$P \subset C$ \defterm{crosses} a wall $w \in W$ if the number of
edges of $P$ colored $w$ is odd.  A combinatorial segment
$P \subset C$ \defterm{begins and ends} with a wall $w \in W$ if the
initial and terminal edges of $P$ map to $w$ under $\alpha$.  Two
distinct walls $w_1,w_2 \in W$ \defterm{cross} if for some
combinatorial segment $P \subset C$, we have that $P$ begins and ends
with one of the two walls and $P$ crosses the other of the two walls.
The \defterm{dimension} $d$ of a wall cycle $(C,\alpha)$ is defined as
$d = \max\{1,n\}$ where $n$ is the size of the largest set
$S \subseteq W$ of pairwise crossing walls.  The \defterm{wall
  crossing distance} $d_{\alpha}(u,v)$ between a pair of vertices
$u,v \in C^0$ is defined as the number of walls crossed by a segment
$P \subset C$ from $u$ to $v$.  Note that the choice of segment $P$
does not matter since each wall appears an even number of times along
$C$.

\begin{prop}
  \proplabel{homcoloring} Let $X$ be a $\CAT(0)$ cube complex.  Let
  $W$ be the set of hyperplanes of $X$ and let
  $\beta \colon X^{(1)} \to W$ map each edge $e$ of $X$ to the
  hyperplane that $e$ crosses.  Then for any cycle
  $f \colon C \to X^1$, the coloring $(C, \alpha)$ is a wall cycle,
  where $\alpha(e) = \beta\bigl(f(e)\bigr)$ for $e \in C^{(1)}$.
  Moreover, two crossing walls of $(C, \alpha)$ must
  cross in $X$ and so the dimension of $X$ is at least the
  dimension of $(C,\alpha)$.  Lastly, the wall crossing distance on
  $(C,\alpha)$ satisfies
  $d_{\alpha}(u,v) = d_{X^1}\bigl(f(u),f(v)\bigr)$.
\end{prop}
\begin{proof}
  That $(C,\alpha)$ is a wall cycle is a consequence of the fact that
  hyperplanes of a $\CAT(0)$ complex are two-sided.  That two crossing
  walls of $(C,\alpha)$ must cross in $X$ is a consequence of the fact
  that hyperplanes are connected and two-sided.  The dimension of a
  $\CAT(0)$ cube complex is equal to the size of the largest set of
  its pairwise crossing hyperplanes.  Finally, the combinatorial
  distance between two vertices of a $\CAT(0)$ cube complex is equal
  to the number of hyperplanes separating them.
\end{proof}

In light of \Propref{homcoloring}, the following theorem implies that
the $1$-skeletons of $d$-dimensional $\CAT(0)$ cube complexes are
strongly shortcut.

\begin{thm}
  \thmlabel{wallcycle} Let $(C,\alpha)$ be a $d$-dimensional wall
  cycle.  If
  $d_{\alpha}(u,v) \ge \bigl(\frac{5d-1}{5d}\bigr)\frac{|C|}{2}$ for
  all antipodal pairs of vertices $u,v \in C^0$ then
  $|C| \le \frac{50d^2}{5d-1}$.
\end{thm}

\begin{cor}
  \corlabel{cccs} The $1$-skeletons of finite dimensional $\CAT(0)$ cube
  complexes are strongly shortcut.
\end{cor}

\begin{cor}
  Cocompactly cubulated groups are strongly shortcut.
\end{cor}

A group is \defterm{cocompactly cubulated} if it acts properly and
cocompactly on a $\CAT(0)$ cube complex.  Such groups include
right-angled Artin groups whose standard Cayley graphs are
$1$-skeletons of $\CAT(0)$ cube complexes
\cite{Charney_Davis:1995:artin} but not all Artin groups are
cocompactly cubulated \cite{haettel:cubulated_artin_tits:2021,
  Huang:2016}.  It is thus natural to ask if all Artin groups are
(strongly) shortcut.

The proof of \Thmref{wallcycle} relies on several lemmas and on the
following theorem of Turan.

\begin{thm}[Turan's Theorem]
  \thmlabel{turan} Let $\Gamma$ be a simplicial graph on $n$ vertices.
  If every complete subgraph of $\Gamma$ has at most $d \in \N$
  vertices then $\Gamma$ has at most
  $\bigl(\frac{d-1}{d}\bigr)\frac{n^2}{2}$ edges.
\end{thm}

Several proofs of Turan's Theorem are given in Aigner and Ziegler
\cite{Aigner:2018}.

\begin{lem}
  \lemlabel{fewappear} Let $(C,\alpha)$ be a wall cycle and suppose
  that for some $\xi \in (0,1)$ we have
  $d_{\alpha}(u,v) \ge \xi \frac{|C|}{2}$ for every antipodal pair
  $u,v \in C^0$.  Let $W' = \{w \in W \sth |\alpha^{-1}(w)| = 2 \}$.
  Then $|W \setminus W'| \le \frac{1 - \xi}{\xi}|W|$.
\end{lem}
\begin{proof}
  Partition $C^{(1)}$ into two sets $S$ and $T$ such that
  $|(\alpha|_S)^{-1}(w)| = |(\alpha|_T)^{-1}(w)|$ for each $w \in W$.
  This is always possible since $|\alpha^{-1}(w)|$ is even for each
  $w \in W$.  Viewing the elements of $S$ as colored by $W$, every
  color appears in $S$ and those colors in $W \setminus W'$ appear at
  least twice.  Hence
  $|W| \le |S| - |W \setminus W'| = \frac{|C|}{2} + |W'| - |W|$.  Let
  $u$ and $v$ be an antipodal pair of vertices.
  Then we have
  \[ \xi\frac{|C|}{2} \le d_{\alpha}(u,v) \le |W| \le \frac{|C|}{2} +
    |W'| - |W| \] and so
  $|W'| \ge |W| - (1 - \xi)\frac{|C|}{2} \ge |W| - \frac{1 -
    \xi}{\xi}|W| = \frac{2\xi - 1}{\xi} |W|$.  Hence we have
  $|W \setminus W'| = |W| - |W'| \le \bigl(1 - \frac{2\xi -
    1}{\xi}\bigr)|W| = \frac{1 - \xi}{\xi}|W|$.
\end{proof}

Let $(C,\alpha)$ be a wall cycle and let $w$ be a wall of
$(C,\alpha)$.  Let $X_w \subset C$ denote the set of all midpoints of
edges colored $w$ and let $\diam X_w$ denote the diameter of $X_w$ as
a metric subspace of $(C,d_C)$.  For a pair of vertices $u,v \in C^0$
we say that $w$ \defterm{contributes} to $\{u,v\}$ if a geodesic
segment from $u$ to $v$ crosses $w$.  Hence $d_{\alpha}(u,v)$ is equal
to the number of walls contributing to $\{u,v\}$.

\begin{lem}
  \lemlabel{diamcontrib} Let $(C,\alpha)$ be a wall cycle and let
  $w \in W$ be a wall such that the number of edges colored $w$ is
  exactly $2$.  Then $w$ contributes to $\{u,v\}$ for exactly
  $\diam X_w$ antipodal pairs of vertices $u,v \in C^0$.
\end{lem}
\begin{proof}
  Let $P \subset C$ be a segment beginning and ending with $w$ of
  length $|P| = \diam X_w + 1$.  Then $w$ contributes to an antipodal
  pair $\{u,v\}$ if and only if one of $u$ or $v$ is an interior
  vertex of $P$ and there are exactly $|P|-1 = \diam X_w$ such pairs.
\end{proof}

\begin{lem}
  \lemlabel{diamints} Let $(C,\alpha)$ be a wall cycle and suppose
  that, for some $\xi \in (0,1)$, we have
  $d_{\alpha}(u,v) \ge \xi\frac{|C|}{2}$ for all antipodal pairs of
  vertices $u,v \in C^0$.  Let $w \in W$ be a wall.  Then $w$
  crosses at least $\diam X_w - 1 - (1-\xi)\frac{|C|}{2}$ walls.
\end{lem}
\begin{proof}
  Consider first the case where $\diam X_w = \frac{|C|}{2}$.  Then
  there exist a pair of antipodal edges $e$ and $e'$ colored $w$.  Let
  $u \in e$ and $u' \in e'$ be a pair of antipodal vertices and let
  $P \subset C$ be a segment with endpoints $u$ and $u'$.  Note that
  $P$ contains exactly one of $e$ or $e'$.  Without loss of generality
  $P$ contains $e$.  Then, since
  $d_{\alpha}(u,u') \ge \xi\frac{|C|}{2}$, then $P$ must cross at
  least $\xi\frac{|C|}{2} - 1$ walls aside from $w$.  Then the same
  must hold for $P \cup e'$ and so $w$ crosses at least
  $\xi \frac{|C|}{2} - 1 = \diam X_w - 1 - (1-\xi)\frac{|C|}{2}$
  walls.

  Consider now the case $\diam X_w < \frac{|C|}{2}$.  We have a
  geodesic segment $P \subset C$ beginning and ending with $w$ such
  that $|P| = \diam X_w + 1$.  Let $u$ and $v$ be the endpoints of
  $P$, let $u'$ be the antipode of $u$ and let $Q$ be the geodesic
  segment containing $P$ and having endpoints $u$ and $u'$.  Then we
  have
  \begin{align*}
    \xi \frac{|C|}{2} \le d_{\alpha}(u,u') &\le
    d_{\alpha}(u,v) + d_{\alpha}(v,u') \\
    &\le d_{\alpha}(u,v) + d_C(v,u') \\
    &= d_{\alpha}(u,v) + \frac{|C|}{2} - (\diam X_w + 1)
  \end{align*}
  and so we have
  $d_{\alpha}(u,v) \ge \diam X_w + 1 - (1 - \xi)\frac{|C|}{2}$.  But
  $w$ crosses at least $d_{\alpha}(u,v) - 1$ walls and so we are
  done.
\end{proof}

\begin{proof}[Proof of \Thmref{wallcycle}]
  Let $\xi = \bigl(\frac{5d-1}{5d}\bigr)$.  For each vertex pair
  $\{u,v\}$ and each wall $w \in W$, let $\mathbbm{1}_w^{\{u,v\}}$ be
  defined as follows.
  \[ \mathbbm{1}_w^{\{u,v\}} =
    \begin{cases}
      1 & \text{if $w$ contributes to $\{u,v\}$} \\
      0 & \text{otherwise}
    \end{cases}
  \]
  Let $W' \subseteq W$ be the set of walls which color exactly two
  edges of $C$.

  We have
  \[ d_{\alpha}(u,v) = \sum_{w \in W} \mathbbm{1}_w^{\{u,v\}}\] and,
  by \Lemref{diamcontrib}, for $w \in W'$ we have
  \[ \diam X_w = \sum_{\{u,v\}\in A} \mathbbm{1}_w^{\{u,v\}} \] where
  $A$ is the set of antipodal pairs of vertices.  Let $\Gamma$ be the
  simplicial graph with vertex set $W$ and where two walls are joined
  by an edge if they cross.  Then we have
  \begin{align*}
    |&\Gamma^{(1)}| \\ &\ge \frac{1}{2}\sum_{w \in W'} \Bigl(\deg(w)\Bigr)
    \\ &\ge \frac{1}{2}\sum_{w \in W'}\Bigl(\diam X_w - 1 - (1-\xi)\frac{|C|}{2}\Bigr) \\
    &= \frac{1}{2}\sum_{w \in W'}\Bigl(\sum_{\{u,v\} \in A}\mathbbm{1}_w^{\{u,v\}}\Bigr) - \frac{1}{2}|W'| - \frac{1}{2}|W'|(1-\xi)\frac{|C|}{2} \\
    &= \frac{1}{2}\sum_{\{u,v\} \in A}\Bigl(\sum_{w \in W}\mathbbm{1}_w^{\{u,v\}}\Bigr) - \frac{1}{2}\sum_{w \in W \setminus W'}\Bigl(\sum_{\{u,v\} \in A}\mathbbm{1}_w^{\{u,v\}}\Bigr) - \frac{1}{2}|W'| - \frac{1}{2}|W'|(1-\xi)\frac{|C|}{2} \\
    &= \frac{1}{2}\sum_{\{u,v\} \in A}\Bigl(d_{\alpha}(u,v)\Bigr)
    - \frac{1}{2}\sum_{w \in W \setminus W'}\Bigl(\sum_{\{u,v\} \in A}\mathbbm{1}_w^{\{u,v\}}\Bigr) - \frac{1}{2}|W'| - \frac{1}{2}|W'|(1-\xi)\frac{|C|}{2} \\
    &\ge \frac{1}{2} \xi \biggl(\frac{|C|}{2}\biggr)^2
    - \frac{1}{2}|W \setminus W'|\cdot\frac{|C|}{2} - \frac{1}{2}|W'| - \frac{1}{2}|W'|(1-\xi)\frac{|C|}{2} \\
    &\ge \frac{1}{2} \xi |W|^2
    - \frac{1}{2} \Bigl(\frac{1 - \xi}{\xi}\Bigr)|W|\cdot\frac{1}{\xi}|W| - \frac{1}{2}|W| - \frac{1}{2}|W|(1-\xi)\frac{1}{\xi}|W| \\
    &= \Bigl(\xi - \frac{1-\xi}{\xi} - \frac{1 - \xi}{\xi^2} - \frac{1}{|W|} \Bigr)\frac{|W|^2}{2} \\
    &= \Bigl(\xi + 1 - \frac{1}{\xi^2} - \frac{1}{|W|} \Bigr)\frac{|W|^2}{2} \\
    &\ge \Bigl( \xi + 1 - \frac{1}{\xi^2} - \frac{2}{\xi|C|} \Bigr)\frac{|W|^2}{2}
  \end{align*}
  where the second inequality holds by \Lemref{diamints} and the
  second to last inequality holds by \Lemref{fewappear}.  We now
  verify that $4x - 3 \le x + 1 - \frac{1}{x^2}$ for
  $x \in \bigl[\frac{4}{5},1\bigr]$, noting that it suffices to check
  the inequality for $x = \frac{4}{5}$ and $x = 1$ since
  $x \mapsto x + 1 - \frac{1}{x^2}$ is a concave function.  Then,
  since $\xi = \frac{5d-1}{5d} \in \bigl[\frac{4}{5},1\bigr]$ we have
  \[ |\Gamma^{(1)}| \ge \Bigl(4\xi - 3 -
    \frac{2}{\xi|C|}\Bigr)\frac{|W|^2}{2} = \Bigl(\frac{4(5d-1)}{5d} -
    3 - \frac{2}{|C|}\cdot\frac{5d}{5d-1}\Bigr)\frac{|W|^2}{2} \] and,
  since $(C,\alpha)$ is $d$-dimensional, we have
  \[ \frac{4(5d-1)}{5d} - 3 - \frac{2}{|C|}\cdot\frac{5d}{5d-1} \le
    \frac{d-1}{d} \] by Turan's Theorem (\Thmref{turan}).  After some
  rearranging and cancellation this inequality becomes
  $|C| \le \frac{50d^2}{5d-1}$.
\end{proof}

\subsection{Cayley graphs of Coxeter groups}

In this section we use the cubulation of Coxeter groups of Niblo and
Reeves \cite{Niblo:2003} and our result on $\CAT(0)$ cube complexes to
prove that Coxeter groups are strongly shortcut.

Let $\Gamma$ be a simplicial graph on the vertex set
$\{v_1, v_2, \ldots, v_n\}$ with every edge labeled by an integer at
least $2$.  If $\Gamma$ has an edge $e$ from $v_i$ to $v_j$ then let
$m_{ij} = m_{ji}$ denote the label of $e$.  The \defterm{Coxeter
  group} $C_{\Gamma}$ defined by $\Gamma$ is given by the following
presentation
\[ \biggpres{v_1, v_2, \ldots, v_n}{\text{$v_k^2=1$ for all $k$ and
      $(v_iv_j)^{m_{ij}}=1$ for all edges
      $\{v_i,v_j\} \in \Gamma^{(1)}$}} \]

For a Coxeter group $C_{\Gamma}$, Niblo and Reeves \cite{Niblo:2003}
construct a finite dimensional $\CAT(0)$ cube complex into whose
$1$-skeleton the Cayley graph $\Cay(C_{\Gamma},\Gamma^0)$
isometrically embeds.  Hence, since the $1$-skeletons of $\CAT(0)$ cube
complexes are strongly shortcut, we have the following theorem.

\begin{thm}
  \thmlabel{coxs} Coxeter groups are strongly shortcut.
\end{thm}

\subsection{Systolic and quadric complexes}

In this section we will prove that the $1$-skeletons of systolic and
quadric complexes are strongly shortcut.  To do so we will rely on
\Corref{cccs}, the characterizations of disk diagrams in systolic and
quadric complexes and a theorem about transforming $2$-dimensional
systolic complexes into quadric complexes.

A \defterm{bridged graph} is a connected graph whose isometric cycles
all have length three \cite{Soltan:1983}.  Bandelt characterized
\defterm{hereditary modular graphs} as those connected graphs whose
isometric cycles all have length four \cite{Bandelt:1988}.  Chepoi
characterized \defterm{systolic complexes} as the flag simplicial
completions of bridged graphs \cite{Chepoi:2000} and we will use this
as a definition here.  The present author characterized
\defterm{quadric complexes} as those square complexes obtained from
hereditary modular graphs by gluing in a square along each embedded
$4$-cycle \cite{Hoda:2017} and we will use this as a definition here.

We require the following lemmas concerning disk diagrams in systolic
and quadric complexes.

\begin{lem}[Chepoi {\cite[Theorem~8.1, Claim~1]{Chepoi:2000}}]
  \lemlabel{systolic_disk_diag} Every cycle in a systolic complex has
  a systolic disk diagram.
\end{lem}

\begin{lem}[{\cite[Lemma~1.6]{Hoda:2017}}]
  \lemlabel{quadric_disk_diag} Every cycle in a quadric complex has a
  $\CAT(0)$ square complex disk diagram.
\end{lem}

The present author and Osajda proved the following theorem
\cite{Hoda:systolic_property_a}, which we also need.

\begin{thm}[{\cite[Theorem~3.2]{Hoda:systolic_property_a}}]
  \thmlabel{squarification} Let $\Gamma$ be the $1$-skeleton of a
  $2$-dimensional systolic complex.  Let $v \in \Gamma^0$ be a vertex.
  Let $\Gamma'$ be the graph obtained from $\Gamma$ by deleting every
  edge whose endpoints are equidistant to $v$.  Then $\Gamma'$ is the
  $1$-skeleton of a quadric complex.
\end{thm}

\begin{lem}
  \lemlabel{cat0ify} Let $D \subset \R^2$ be a systolic disk diagram.
  Let $v \in \bd D$ be a vertex on the boundary of $D$.  Let $\Gamma$
  be the plane graph obtained from the $1$-skeleton $D^1$ by deleting
  every edge whose endpoints are equidistant to $v$.  Then the square
  complex $D_{-} \subset \R^2$ obtained from $\Gamma$ by including any
  planar region bounded by an embedded $4$-cycle of $\Gamma$ is a
  $\CAT(0)$ square complex disk diagram.
\end{lem}
\begin{proof}
  By \Thmref{squarification}, so long as every embedded $4$-cycle of
  $\Gamma$ bounds a planar region of $\R^2$, the square complex
  $D_{-}$ is quadric.  But quadric complexes are simply connected, by
  \Lemref{quadric_disk_diag} and planar quadric complexes are
  $\CAT(0)$ so proving that every embedded $4$-cycle $C$ of $\Gamma$
  bounds a planar region of $\R^2$ will suffice to prove the lemma.
  But such a $C$ is an embedded $4$-cycle of $D^1$ which is a bridged
  graph so some antipodal pair of vertices of $C$ are joined by an
  edge $e$ in $D^1$ making $C \cup \{e\}$ the union of two $3$-cycles
  $C_1$ and $C_2$ with $C_1 \cap C_2 = \{e\}$.  Then, since $D$ is a
  flag simplicial complex, the $3$-cycles $C_1$ and $C_2$ bound planar
  regions and thus so does $C$.
\end{proof}

Let $D \subset \R^2$ be a simplicial disk diagram whose boundary
$\bd D$ is an embedded cycle.  Let Let $v \in \bd D$ be a vertex and
let $E$ be the set of edges of $\bd D$ whose endpoints are equidistant
to $v$ in $D^1$.  Let $D_{+} \subset \R^2$ be the disk diagram
obtained from $D$ by adding a new triangle $T_e$ to $D$ for each
$e \in E$ by identifying $e$ with an edge of $T_e$.  Notice that $D^1$
is convex in $D_{+}^1$.  Thus no boundary edge of $D_{+}$ has
endpoints that are equidistant to $v$ in $D_{+}^1$ and if $D$ is
systolic then so is $D_{+}$.

\begin{lem}
  \lemlabel{adding_ears} Let $D \subset \R^2$ be a systolic disk
  diagram whose boundary $\bd D$ is a $K$-bilipschitz embedded cycle
  $\bd D$ of $D^1$.  Let $v \in \bd D$ be a vertex and let $E$ be the
  set of edges of $\bd D$ whose endpoints are equidistant to $v$ in
  $D^1$.  Then at most $|E| \le \frac{K-1}{K}\cdot |\bd D| + 1$.

  Consequently, if $|\bd D| \ge \frac{2K(2K + 3)}{K-1}$ then the boundary
  $\bd D_{+}$ of the disk diagram $D_{+} \subset \R^2$ obtained from
  $D$ as above is an embedded $\frac{3-2K}{K^2}$-almost isometric
  cycle.
\end{lem}
\begin{proof}
  One of the two embedded paths in $\bd D$ from $v$ to its antipode
  $\bar v$ contains at least $\bigl\lceil\frac{|E|-1}{2}\bigr\rceil$
  edges of $E$.  But then
  $d_{D^1}(v, \bar v) \le \frac{|\bd D|}{2} -
  \bigl\lceil\frac{|E|-1}{2}\bigr\rceil \le \frac{|\bd D|}{2} -
  \frac{|E|-1}{2}$, which by $K$-bilipschitz embeddedness implies
  $\frac{1}{K}\cdot\frac{|\bd D|}{2} \le \frac{|\bd D|}{2} -
  \frac{|E|-1}{2}$.  After rearranging we obtain
  $|E| \le \frac{K-1}{K}\cdot |\bd D| + 1$, as required.

  Now let $D_{+} \subset \R^2$ be obtained from $D$ as above.  Let
  $p, \bar p \in \bd D_{+}$ be antipodal.  There exist vertices
  $u, v \in \bd D \cap \bd D_{+}$ with $d_{\bd D_{+}}(u,p) \le 1$ and
  $d_{\bd D_{+}}(v,\bar p) \le 1$.  Then
  $d_{\bd D_{+}}(u,v) \ge \frac{|\bd D_{+}|}{2} - 2$ but
  $d_{\bd D_{+}}(u,v) \le d_{\bd D}(u,v) + |E| \le d_{\bd D}(u,v) +
  \frac{K-1}{K}\cdot |\bd D| + 1 \le d_{\bd D}(u,v) +
  \frac{K-1}{K}\cdot |\bd D_{+}| + 1$ so
  $d_{\bd D}(u,v) \ge \frac{|\bd D_{+}|}{2} - 2 - \frac{K-1}{K}\cdot
  |\bd D_{+}| - 1 = \frac{2 - K}{2K} \cdot |\bd D_{+}| - 3$.  Hence by
  $K$-bilipschitz embeddedness of $\bd D$ in $D^1$, by convexity of
  $D^1$ in $D_{+}^1$ and by
  $|\bd D_{+}| \ge |\bd D| \ge \frac{2K(2K + 3)}{K-1}$, we have
  \begin{align*}
    d_{D_{+}^1}(p, \bar p)
    &\ge d_{D_{+}^1}(u,v) - 2 \\
    &= d_{D^1}(u,v) - 2 \\
    &\ge \frac{1}{K} d_{\bd D}(u,v) - 2 \\
    &\ge \frac{2 - K}{2K^2} \cdot |\bd D_{+}| - \frac{3}{K} - 2 \\
    &= \frac{2 - K}{2K^2} \cdot |\bd D_{+}| - \frac{K(2K + 3)}{K^2} \\
    &= \Bigl(\frac{2 - K}{K^2} - \frac{2K(2K + 3)}{K^2|D_{+}|}\Bigr) \cdot
      \frac{|\bd D_{+}|}{2} \\
    &\ge \Bigl(\frac{2 - K}{K^2} - \frac{K-1}{K^2}\Bigr) \cdot
      \frac{|\bd D_{+}|}{2} \\
    &= \frac{3 - 2K}{K^2} \cdot \frac{|\bd D_{+}|}{2}
  \end{align*}
  as required.
\end{proof}

\begin{thm}
  \thmlabel{sys_quad_ss} The $1$-skeletons of systolic and quadric
  complexes are strongly shortcut.
\end{thm}
\begin{proof}
  Let $X$ be a systolic or quadric complex.  For the sake of finding a
  contradiction, assume that $X^1$ is not strongly shortcut.  Then
  there exists a sequence $(f_n \colon C_n \to X^1)_{n \in \N}$ of
  cycles such that $f_n$ is an $\frac{n+1}{n}$-bilipschitz embedding
  and $|C_n| \ge n$.  We will use the $f_n$ to construct a sequence of
  spaces $(D_n)_n$ such that
  \begin{enumerate}
  \item $D_n \subset \R^2$ is a $\CAT(0)$ square complex disk diagram, and
  \item $\bd D_n$ is an embedded $\xi_n$-almost isometric cycle in
    $D_n^1$,
  \end{enumerate}
  where $|\bd D_n| \to \infty$ and $\xi_n \to 1$ as $n \to \infty$.
  This will suffice since, by taking the wedge of the $D_n$, we
  contradict \Corref{cccs}.

  In the case where $X$ is quadric, let $\bar f_n \colon D_n \to X$ be
  the $\CAT(0)$ square complex disk diagram for $f_n$ guaranteed by
  \Lemref{quadric_disk_diag}.  Then, since $f_n = \bar f_n|_{\bd D_n}$
  is an $\frac{n+1}{n}$-bilipschitz embedding and $\bar f_n|_{D_n^1}$
  is $1$-Lipschitz, the inclusion $\bd D_n \hookrightarrow D_n^1$ is
  $\frac{n+1}{n}$-bilipschitz and so is an $\frac{n}{n+1}$-almost
  isometric embedding of a cycle of length $|\bd D_n| = |C_n| \ge n$.

  In the case where $X$ is systolic, we will need a slightly more
  sophisticated argument.  Pick a subsequence
  $(f_{n_k} \colon C_{n_k} \to X^1)_{k \in \N}$ such that $f_{n_k}$ is
  $K_k$-bilipschitz with $K_k = \frac{k+1}{k}$ and
  $|C_{n_k}| \ge \frac{2K_k(2K_k+3)}{K_k-1}$.  Let
  $\bar f_{n_k} \colon D_{n_k} \to X$ be the systolic disk diagram for
  $f_n$ guaranteed by \Lemref{systolic_disk_diag}.  Then $D_{n_k}$
  satisfies the conditions of \Lemref{adding_ears} with $K = K_k$,
  thus we obtain a systolic disk diagram $D_{n_k,+} \subset \R^2$ from
  $D_{n_k}$, as above, whose boundary $\bd D_{n_k,+}$ is an embedded
  $\frac{3-2K_k}{K_k^2}$-almost isometric cycle of length
  $|\bd D_{n_k,+}| \ge |\bd D_{n_k}| = |C_{n_k}| \ge
  \frac{2K_k(2K_k+3)}{K_k-1}$.  Moreover, for some vertex
  $v_k \in \bd D_{n_k,+}$ no edge $e$ of $\bd D_{n_k,+}$ has endpoints
  that are equidistant to $v_k$ in $D_{n_k,+}^1$.

  Let $D_{n_k,-} \subset \R^2$ be the $\CAT(0)$ square complex disk
  diagram obtained from $D_{n_k,+}$ as in the statement of
  \Lemref{cat0ify} with $v$ set to $v_k$.  Then
  $\bd D_{n_k,-} = \bd D_{n_k,+}$ and
  $D_{n_k,-}^1 \subset D_{n_k,+}^1$ so $\bd D_{n_k,-}$ is an embedded
  $\frac{3-2K_k}{K_k^2}$-almost isometric cycle of length
  $|\bd D_{n_k,-}| \ge \frac{2K_k(2K_k+3)}{K_k-1}$ in $D_{n_k,-}$ but
  $K_k \to 1$ so $\frac{3-2K_k}{K_k^2} \to 1$ and
  $\frac{2K_k(2K_k+3)}{K_k-1} \to \infty$ as $k \to \infty$, thus the
  $(D_{n_k,-})_k$ are as required.
\end{proof}

\begin{cor}
  \corlabel{sys_quad_is_shortcut} Systolic and quadric groups are
  strongly shortcut.
\end{cor}

Wise proved that finitely presented $C(6)$ small cancellation groups
are systolic \cite{Wise:2003} and the present author proved that
finitely presented $C(4)$-$T(4)$ small cancellation groups are quadric
\cite{Hoda:2017} so we have the following corollary.

\begin{cor}
  \corlabel{sc_is_shortcut} Finitely presented $C(6)$ and
  $C(4)$-$T(4)$ small cancellation groups are strongly shortcut.
\end{cor}

\subsection{Cayley graphs of \texorpdfstring{$\Z$}{Z} and
  \texorpdfstring{$\Z^2$}{Z\texttwosuperior}}

We have shown that the $1$-skeletons of $\CAT(0)$ cube complexes are
strongly shortcut.  In particular, the standard Cayley graphs of the
finitely generated free abelian groups are strongly shortcut.  In this
section we will strengthen this result for $\Z$ and $\Z^2$ by showing
that all of their Cayley graphs are strongly shortcut.

\begin{lem}
  \lemlabel{useplank} Let $\Gamma$ be a graph and suppose there is a
  continuous $(K,M)$-quasi-isometric embedding
  $\iota \colon \Gamma \to \R^2$.  Let $\xi \in (0,1)$ and let
  $f \colon C \to \Gamma $ be a $\xi$-almost isometric cycle.  Suppose
  the image of $\iota \comp f$ is contained in the $N$-neighborhood
  of a line $L \subset \R$.  Then $|C| \le \frac{2K}{\xi}(M+2N)$.
\end{lem}
\begin{proof}
  By continuity, for some pair of antipodal points $p,q \in C$, the
  points $\iota \comp f(p)$ and $\iota \comp f(q)$ project
  perpendicularly to the same point of $L$.  Then
  \[ \frac{1}{K} d_{\Gamma}\bigl(f(p), f(q)\bigr) - M \le d_{\R}\bigl(\iota \comp f(p), \iota \comp f(q)\bigr) \le 2N \]
  and so we have
  \[ \xi \frac{|C|}{2} \le d_{\Gamma}\bigl(f(p),f(q)\bigr) \le
    K(M+2N) \] and so we have $|C| \le \frac{2K}{\xi}(M+2N)$.
\end{proof}

Since the inclusion map $\Z \hookrightarrow \R \times \{0\}$ extends
to a continuous quasi-isometric embedding from any Cayley graph of
$\Z$, we obtain as a corollary of \Lemref{useplank} the following
theorem.

\begin{thm}
  \thmlabel{cayz} Every Cayley graph of $\Z$ is strongly shortcut.
\end{thm}

The Cayley graphs of $\Z$ are all quasi-isometric to $\R$ and so are
hyperbolic.  Thus \Thmref{cayz} also follows from \Thmref{hypgraph}.
In the remainder of this section we will prove the strong shortcut
property for Cayley graphs of $\Z^2$ where we cannot rely on
hyperbolicity.  In fact, we cannot even rely on the quasi-isometry
type of $\Z^2$ as the following example makes clear.

\begin{ex}\exlabel{not_qi_invariant}
  Let $\Gamma$ be the standard Cayley graph of $\Z^2$.  For each
  $n \in \N$, let $A_n$ be the induced subgraph on
  $\{0, 1, \ldots, n\}^2$, let $P_n$ be the induced subgraph on
  $\{0, 1, \ldots, n\} \times \{0\}$ and let $Q_n$ be the induced
  subgraph on $\{0\} \times \{0, 1, \ldots, n\}$.  Then
  $C_n = P_n \cup Q_n \cup \bigl(P_n + (0,n)\bigr) \cup \bigl(Q_n +
  (n,0)\bigr)$ is the embedded cycle that ``bounds'' $A_n$.  Note that
  the $A'_n = A_n + \Bigl(\frac{n(n+1)}{2},\frac{n(n+1)}{2}\Bigr)$ are
  disjoint.  Let
  $C'_n = C_n + \Bigl(\frac{n(n+1)}{2},\frac{n(n+1)}{2}\Bigr)$ and let
  $\Gamma'$ be the graph obtained from $\Gamma$ by subdividing the
  edges of each $A'_n \setminus C'_n$.  Then the $C'_n$ are
  isometrically embedded in $\Gamma'$.  Thus $\Gamma$ is not shortcut
  and yet $\Gamma'$ is quasi-isometric to $\Gamma$ and thus to $\R^2$.
\end{ex}

Let $S$ be a generating set of $\Z^2$, let $\Gamma$ be the Cayley
graph of $(\Z^2,S)$ and let $\iota \colon \Gamma \to \R^2$ be the
$(K,M)$-quasi-isometry obtained by extending the inclusion map
$\Z^2 \hookrightarrow \R^2$ to $\Gamma$ in such a way that the
restriction of $\iota$ to each edge is a constant-speed geodesic.

\begin{lem}
  \lemlabel{getplank} Let $f \colon C \to \Gamma$ be a $\xi$-almost
  isometric embedding.  For some constants $A$ and $B$ depending only
  on $S$, there is a line in $\R^2$ whose
  $\bigl((1-\xi)A|C| + B\bigr)$-neighborhood contains the image of
  $\iota \comp f$.
\end{lem}
\begin{proof}
  For $x \in \R^2$ let $|x|$ denote the standard Euclidean norm of
  $x$.  Let $t \in S$ achieve
  $|t| = \max\bigl\{|s| \sth s \in S\bigr\}$.  Let $V$ be the
  $1$-dimensional vector subspace of $\R^2$ generated by $t$.  For
  $s \in S$ let $s_t$ be the perpendicular projection of $s$ onto $V$
  and let
  $\alpha = \max\bigl\{|s_t| \sth s \in S \setminus \{\pm t\}\bigr\}$.
  Then $\alpha < |t|$.

  By continuity, some pair of antipodal points $p,q \in C$ satisfy
  $\iota \comp f(p) - \iota \comp f(q) \in V$.  Pick $u \in \Z^2$ such
  that $|u-\iota\comp f(p)| \le 1$.  Then for some $r \in \Z$, we have
  $\bigl|u+rt - \iota\comp f(q)\bigr| \le 1 + |t|$.  Then
  $d_{\Gamma}\bigl(u,f(p)\bigr) \le K(M+1)$ and
  $d_{\Gamma}\bigl(u+rt,f(q)\bigr) \le K\bigl(M + 1 +|t|\bigr)$.
  Hence we have
  \begin{align*}
    \xi \frac{|C|}{2} &\le d_{\Gamma}\bigl(f(p),f(q)\bigr) \\
                      &\le d_{\Gamma}\bigl(f(p), u\bigr)
    + d_{\Gamma}(u,u+rt) + d_{\Gamma}\bigl(u+rt,f(q)\bigr) \\
    &\le |r| + K\bigl(2M+2 + |t|\bigr)
  \end{align*}
  and so $|r| \ge \xi\frac{|C|}{2} - K\bigl(2M+2 + |t|\bigr)$.  Let
  $P \subset C$ be a segment with endpoints $p$ and $q$.  Each edge of
  $\Gamma$ is labeled by a generator in $S$.  Pull these labels back
  to $C$ under $f$.  Let $T$ be the union of all the $t$-labeled edges
  of $C$ and let $\ell$ be the total length of the segments of
  $T \cap P$.  Consider the projection of the path $\iota \comp f|_P$
  onto the line $\iota \comp f(p) + V$.  It has arclength at most
  $\ell|t| + \bigl(\frac{|C|}{2} - \ell\bigr)\alpha$.  But the
  endpoints of $\iota \comp f|_P$ are $\iota \comp f(p)$ and
  $\iota \comp f(q)$, which are of distance at least
  $\bigl(|r|-1\bigr)|t| - 2$ apart and so
  $\bigl(|r|-1\bigr)|t| - 2 \le \ell|t| + \bigl(\frac{|C|}{2} -
  \ell\bigr)\alpha$.  Combining this inequality with
  $|r| \ge \xi\frac{|C|}{2} - K\bigl(2M+2 + |t|\bigr)$ we have
  \[ \biggl(\xi \frac{|C|}{2} - K\bigl(2M + 2 + |t|\bigr)-1\biggr)|t| -
    2 \le \ell|t| + \biggl(\frac{|C|}{2} - \ell\biggr)\alpha \] which,
  after some manipulation gives
  \[ \ell \ge \biggl(\frac{\xi|t| - \alpha}{|t| -
      \alpha}\biggr)\frac{|C|}{2} - \frac{K\bigl(2M+2+|t|\bigr)|t| +
      |t| + 2}{|t| - \alpha} \] and so we have the following:
  \[ |P \setminus T| = \frac{|C|}{2} - \ell \le \biggl(\frac{(1 -
      \xi)|t|}{|t| - \alpha}\biggr)\frac{|C|}{2} +
    \frac{K\bigl(2M+2+|t|\bigr)|t| + |t| + 2}{|t| - \alpha} \] But
  then the projection of $\iota \comp \alpha|_P$ to $V^{\perp}$ must
  have length at most
  \[ \biggl(\frac{(1 - \xi)|t|^2}{|t| - \alpha}\biggr)\frac{|C|}{2} +
    \frac{K\bigl(2M+2+|t|\bigr)|t|^2 + |t|^2 + 2|t|}{|t| - \alpha} \]
  and so the image of $\iota \comp f|_P$ must be contained in a
  neighborhood of radius
  \[ \biggl(\frac{(1 - \xi)|t|^2}{2\bigl(|t| -
      \alpha\bigr)}\biggr)\frac{|C|}{2} +
    \frac{K\bigl(2M+2+|t|\bigr)|t|^2 + |t|^2 + 2|t|}{2\bigl(|t| -
      \alpha\bigr)} \] about the line $\iota \comp f(p) + V$.  Then
  the lemma holds with $A = \frac{|t|^2}{4(|t| - \alpha)}$
  and $B = \frac{K(2M+2+|t|)|t|^2 + |t|^2 + 2|t|}{2(|t| - \alpha)}$.
\end{proof}

\begin{thm}
  \thmlabel{cayzz} Every Cayley graph of $\Z^2$ is strongly shortcut.
\end{thm}
\begin{proof}
  Let $f \colon C \to \Gamma$ be a $\xi$-almost isometric cycle.  By
  \Lemref{getplank} we have a line $L \subset \R^2$ whose
  $\bigl((1-\xi)A|C| + B\bigr)$-neighborhood contains the image of
  $\iota \comp f$.  So, by \Lemref{useplank}, we have
  \[ |C| \le \frac{2K}{\xi}\bigl(M+2(1-\xi)A|C| + 2B\bigr) \] and
  so
  \[ \biggl(1 - \frac{4K}{\xi}(1-\xi)A\biggr)|C| \le
    \frac{2K}{\xi}(M+2B) \] which gives us a bound on the length of
  $|C|$ assuming we have $1 - \frac{4K}{\xi}(1-\xi)A > 0$.  But this
  condition is equivalent to $\xi > \frac{4KA}{1 + 4KA}$.  Hence, for
  $\xi \in \bigl(\frac{4KA}{1 + 4KA}, 1\bigr)$, there is a bound on
  the length of the $\xi$-almost isometric cycles of $\Gamma$.
\end{proof}

\section{The Baumslag-Solitar group
  \texorpdfstring{$\BS(1,2)$}{BS(1,2)}}

\seclabel{bsgp}

The Baumslag-Solitar group $\BS(1,2)$ is defined by the following
presentation. \[ \gpres{a,t}{tat^{-1} = a^2} \] In this section we
will show that the standard Cayley graph of $G = \BS(1,2)$ is shortcut
but that adding the generator $\tau = t^2$ results in a Cayley graph
$\Cay\bigl(G,\{a,t,\tau\}\bigr)$ which is not shortcut.  Hence we see
that there exists a shortcut group with exponential Dehn function
\cite{Gersten:1992} and that the shortcut property for a Cayley graph
is not invariant under a change of generating set.  We also see that
there exists a shortcut group which is not strongly shortcut, since
strongly shortcut groups have polynomial isoperimetric function, by
\Corref{isoper}.

Let $\Gamma$ be the Cayley graph of $\BS(1,2)$ with generating set
$\{a,t\}$.  Since $\BS(1,2)$ is an HNN extension it has a Bass-Serre
tree $T$.  Every vertex of $T$ has two outgoing edges labeled $t$ and
one incoming edge labeled $t$.

\begin{lem}
  \lemlabel{normform} Every element of $\BS(1,2)$ can be written
  uniquely in the form $t^ma^kt^n$ where $m,k,n \in \Z$ and $k$ is
  even only if $k = m = 0$.
\end{lem}
\begin{proof}
  Given any word representing an element of $\BS(1,2)$ in the standard
  generators, we may commute positive powers of $t$ to the right and
  negative powers of $t$ to the left using the relations
  $t^na^k = a^{2^nk}t^n$ and $a^kt^{-n} = t^{-n}a^{2^nk}$, with
  $n \ge 0$, to obtain a representative of the form $t^ma^kt^n$.  Then
  we may apply the relation $a^{k} = t^na^{k/2^n}t^{-n}$ if $k$ is a
  nonzero integer multiple of $2^n$, with $n \ge 0$, to obtain a
  representative of the form $t^ma^kt^n$ where $k$ is even if and only
  if $k = m= 0$.

  To see that this form is unique, let
  $t^{m'} a^{k'} t^{n'} = t^m a^k t^n$.  By consideration of the
  Bass-Serre tree $T$ we must have $m+n = m'+n'$.  Without loss of
  generality $m \ge m'$ and so we have
  \[ a^{k'}t^{n'} = t^{m - m'}a^kt^n = a^{2^{m-m'}k}t^{m-m'}t^n =
    a^{2^{m-m'}k}t^{n'} \] and so, as the base group embeds in an HNN
  extension, we have $k' = 2^{m-m'}k$.  So, in the case where $k=m=0$,
  we have $k' = 0$ and so $m' = 0$, which implies
  $(m',k',n') = (m,k,n)$.  If $k \neq 0$ then $k' \neq 0$ and so $k$ and
  $k'$ are both odd integers.  Hence $2^{m-m'} = 1$, which again
  implies $(m',k',n') = (m,k,n)$.
\end{proof}

It follows from \Lemref{normform} that we have a one-to-one
correspondence
\begin{align*}
  \phi \colon G &\to \Z\Bigl[\frac{1}{2}\Bigr]\times\Z \\
  t^ma^kt^n &\mapsto (2^mk, m+n)
\end{align*}
with inverse
\begin{align*}
  \phi^{-1} \colon \Z\biggl[\frac{1}{2}\biggr]\times\Z &\to G \\
  (r, z) &\mapsto
           \begin{cases}
             t^{\nu(r)}a^{r/2^{\nu(r)}}t^{z - \nu(r)} &\text{if $r \neq 0$} \\
             t^z & \text{if $r = 0$}
           \end{cases} \\
\end{align*}
where $\Z\bigl[\frac{1}{2}\bigr]$ is the set of dyadic rationals and
$\nu(r)$ is defined as follows.
\[ \nu(r) = \max\Bigl\{m \in \Z \sth \text{$r$ is an integer multiple
    of $2^m$}\Bigr\} \] The \defterm{height} of a point $(r,z)$ is
$z$.  We use $\mu(g)$ to denote the height of $\phi(g)$ for
$g \in \BS(1,2)$.  For a word $w$ in a generating of $\BS(1,2)$ we
define the \defterm{height} $\mu(w)$ of $w$ as $\mu(g)$ for the
element $g \in \BS(1,2)$ that is represented by the word $w$.

Pushing forward the group operation to
$\Z\bigl[\frac{1}{2}\bigr]\times\Z$ gives the following operation.
\[ (r,z)\cdot(r',z') = (r + 2^{z}r',z+z') \] Pushing forward the
Cayley graph structure gives the following edges.
\begin{align*}
  (r,z) &\xrightarrow{a} (r+2^z,z) \\
  (r,z) &\xrightarrow{t} (r,z+1)
\end{align*}
The Bass-Serre tree $T$ may be identified with the quotient of this
graph which identifies $(r,z)$ and $(r',z)$ if $r-r'$ is an integer
multiple of $2^z$.  This identification preserves height so we may
refer to the \defterm{height} $\mu(v)$ of a vertex $v$ of $T$.

\begin{lem}
  \lemlabel{height_reducing_path_parent} Let $u$ and $v$ be distinct vertices
  of $T$ with $\mu(u) \ge \mu(v)$ and let $u^{(-1)} \xrightarrow{t} u$
  be the unique incoming edge of $T$ at $u$.  If
  $(u = u_0, u_1, \ldots, u_k = v)$ is the sequence of vertices of a
  path from $u$ to $v$ in $T$ then $u_i = u^{(-1)}$ for some
  $i \in \{1, \ldots k\}$.
\end{lem}
\begin{proof}
  The proof is by induction on $k$.  Since $u \neq v$ we have
  $k \ge 1$.  If $k = 1$ then there must be an edge joining $u$ and
  $v$ but $u^{(-1)}$ is the only neighbor of $u$ whose height is not
  greater than that of $u$ so $u_1 = v = u^{(-1)}$.

  Suppose $k > 1$.  Let $i \in \{0,1, \ldots, k\}$ be maximal such
  that $u_i = u$.  Since $u \neq v$ we have $i < k$.  We claim that
  $u_{i+1} = u^{(-1)}$.  Indeed, if this were not the case then we would
  have $u \xrightarrow{t} u_{i+1}$ so that
  $\mu(u_{i+1}) > \mu(u) \ge \mu(v)$ and the induction hypothesis
  applied $(u_{i+1}, u_{i+2}, \ldots, u_k = v)$ to would imply that
  $u_j = u$ for some $j \in \{i+2, i+3, \ldots, k\}$ contradicting
  maximality of $i$.
\end{proof}

\begin{lem}
  \lemlabel{height_reducing_path} Let $u,v \in T$ with
  $\mu(u) - \mu(v) = h > 0$ and let
  \[ u^{(-h)} \xrightarrow{t} u^{(-h+1)} \xrightarrow{t} \cdots
    \xrightarrow{t} u^{(-1)} \xrightarrow{t} u \] be the unique
  directed path of length $h$ ending at $u$ in $T$.  If
  $(u = u_0, u_1, \ldots, u_k = v)$ is the sequence of vertices of a
  path from $u$ to $v$ in $T$ then $u_i = u^{(-h)}$ for some
  $i \in \{1, 2, \ldots k\}$.
\end{lem}
\begin{proof}
  The proof is by induction on $h$.  By
  \Lemref{height_reducing_path_parent}, we have $u_i = u^{(-1)}$ for
  some $i \in \{1, 2, \ldots, k\}$.  This proves the base case
  $h = 1$.  For $h > 1$ the lemma follows from the induction
  hypothesis applied to $(u_i, u_{i+1}, \ldots, u_k = v)$.
\end{proof}

\begin{lem}
  \lemlabel{minheight} Every finite subtree of $T$ has a unique vertex
  of minimal height.
\end{lem}
\begin{proof}
  This follows from the fact that each vertex of $T$ has two outgoing
  edges and one incoming edge and that height increases when
  traversing the outgoing edges and decreases when traversing the
  incoming edges.  Let $T'$ be a finite subtree and suppose, for the
  sake of finding a contradiction, that $v$ and $v'$ are distinct
  minimal height vertices of $T'$.  Let $P$ be the unique embedded
  path from $v$ to $v'$ in $T$.  Then $P \subset T'$ and so, by
  minimality of the height of $v$, the first edge of $P$ is outgoing
  from $v$.  Then, since $P$ is an embedded path and each vertex of
  $T$ has at most one incoming edge, every subsequent edge of $P$ is
  also directed away from $v$ and towards $v'$.  Thus $v'$ has greater
  height than $v$, a contradiction.
\end{proof}

\subsection{Geodesics in \texorpdfstring{$\BS(1,2)$}{BS(1,2)}}

We now prove some lemmas about geodesics in the Cayley graph $\Gamma$.
We will describe paths in $\Gamma$ using words in the letters
$\{a^{\pm 1},t^{\pm 1}\}$.  We use the notation $w_1 \equiv w_2$ to
mean that the words $w_1$ and $w_2$ represent the same element of
$\BS(1,2)$.  Of course if $w_1 \equiv w_2$ then paths described by
$w_1$ and $w_2$ and starting at the same initial vertex must have the
same final vertex.

\begin{rmk}
  A path in $\Gamma$ may be projected to a path in $T$.  A backtrack
  in the projection corresponds to a subword of the form $ta^kt^{-1}$
  or $t^{-1}a^{2k}t$.  The initial and terminal edges of a path
  described by $twt^{-1}$ project to the same edge in $T$ if and only
  if $w \equiv a^k$.  The initial and terminal edges of a path
  described by $t^{-1}wt$ project to the same edge in $T$ if and only
  if $w \equiv a^{2k}$.
\end{rmk}

\begin{lem}
  \lemlabel{notgeo} Words of the following forms do not describe
  geodesic paths.
  \begin{enumerate}
  \item \itmlabel{tati} $ta^{\pm 1}t^{-1}$
  \item \itmlabel{tiak} $t^{-1}a^k$ and $a^kt$ with $|k| \ge 2$
  \item \itmlabel{rightdownleft} $a^{\epsilon}t^{-1}a^{-\epsilon}$
    with $\epsilon = \pm 1$.
  \item \itmlabel{downupdown} $t^{-1}w_1tw_2t^{-1}$ with
    $w_1 \equiv a^{k_1}$ and $w_2 \equiv a^{k_2}$
  \item \itmlabel{tconvex} $w \equiv t^h$ with $w \neq t^h$.
  \item \itmlabel{downup} $w_1w_2$ with $w_1w_2 \equiv a^k$ and
    where $\mu(w_1) < 0$
  \end{enumerate}
\end{lem}
\begin{proof} The following equivalences prove nongeodesicity for
  \pitmref{tati}, \pitmref{tiak}, \pitmref{rightdownleft} and
  \pitmref{downupdown}.
  \begin{itemize}
  \item[\pitmref{tati}] $ta^{\pm 1}t^{-1} \equiv a^{\pm 2}$
  \item[\pitmref{tiak}] $t^{-1}a^k \equiv a^{\pm 1}t^{-1}a^{k \mp 2}$
  \item[\pitmref{rightdownleft}]
    $a^{\epsilon}t^{-1}a^{-\epsilon} \equiv t^{-1}a^{\epsilon}$
  \item[\pitmref{downupdown}]
    $t^{-1}w_1tw_2t^{-1} \equiv t^{-1}a^{k_1}a^{2k_2} \equiv
    t^{-1}a^{2k_2}a^{k_1} \equiv t^{-1}tw_2t^{-1}w_1 \equiv w_2t^{-1}w_1$
  \end{itemize}
  
  \pitmref{tconvex} Suppose $w$ is geodesic with $w \equiv t^h$.  Note
  that $\mu(w) = h$ so $w$ must contain at least $|h|$ instances of
  $t^{\epsilon}$ where $\epsilon$ is the sign of $h$.  Hence, as
  $|w| \le |t^h| = h$, we see that $w$ cannot contain any instance of
  $a^{\pm 1}$.  But $w$ may not contain any backtracks either and so
  $w = t^h$.

  \pitmref{downup} Suppose $w = w_1w_2$ is geodesic with
  $w \equiv a^k$ and $\mu(w_1) < 0$.  We view $w$ as a path
  $P \to \Gamma$.  We pull back labels and directions from $\Gamma$ so
  that the edges of $P$ are directed and labeled with the generators
  $a$ and $t$.  In this way, each subpath of $P$ is labeled by a word
  in $a$, $t$ and their inverses.  By \Lemref{minheight}, there is a
  unique vertex $v$ of minimal height of the projection of $w$ to the
  Bass-Serre tree $T$ and $\mu(v) \le \mu(w_1) < 0 = \mu(1) = \mu(w)$.
  Then $w$ must contain a subpath labeled $t^{-1}a^{\ell}t$ with the
  $a^{\ell}$ part mapping to $v$ under the projection to $T$.  Then
  \pitmref{tiak} implies that $|k| = 1$.  So
  $t^{-1}a^{\ell}t=t^{-1}a^{\pm 1}t$ corresponds to a nonbacktracking
  path in $T$.  Since the projection of $w$ to $T$ is a closed path
  and $v$ is a cutpoint of $T$, it follows that $w$ contains another
  subpath labeled $t^{-1}a^{\pm 1}t$ such that $a^{\pm 1}$ maps to $v$
  under the projection to $T$.  Then $w$ contains a subword labeled
  $t^{-1}a^{\epsilon_1}tw't^{-1}a^{\epsilon_2}t$ with
  $\epsilon_i \in \{\pm 1\}$ and $w' \equiv a^{k'}$.  But, by
  \pitmref{downupdown}, we know that $t^{-1}a^{\epsilon_1}tw't^{-1}$
  is not geodesic, which is a contradiction.
\end{proof}

\begin{lem}
  \lemlabel{zeroheight} If $\mu(w) = 0$ and every prefix $w'$ of $w$
  has $\mu(w') \ge 0$ then $w \equiv a^k$ for some $k$.
\end{lem}
\begin{proof}
  Viewing $w$ as a path in $\Z\bigl[\frac{1}{2}\bigr]\times\Z$
  starting at $(0,0)$, we see that at each step the first coordinate
  changes by a positive power of $2$.  Hence the endpoint of the path
  is $(k,0)$ for some integer $k$.
\end{proof}

A word $w$ is \defterm{ascending} if it contains only positive powers
of $t$ and \defterm{descending} if it contains only negative powers of
$t$.

\begin{lem}
  \lemlabel{descent} Let $w$ be a geodesic word with
  $w \equiv t^{-h}a^k$ where $h \ge 0$.  Then no prefix $w'$ of $w$
  satisfies $\mu(w') < -h$ and we have $w = xy$ where $x$ is ascending
  and $y$ is descending.
\end{lem}
\begin{proof}
  Let $w = w_1w_2$ where $w_1$ is the smallest prefix of $w$ with
  $\mu(w_1) = -h$.  View $w_1$ as a path $f \colon P \to \Gamma$ and
  let $\bar f \colon P \to T$ be the projection of this path to $T$.
  Then $\bar f$ starts at the vertex of $T$ corresponding to the coset
  $\subgp{a}$ and $\bar f$ ends at a vertex of height $-h$ and
  $\bar f$ does not reach the height $-h$ until its final step.  Thus,
  by \Lemref{height_reducing_path}, the final vertex of $\bar f$ is
  the vertex of $T$ corresponding to the coset $t^{-h}\subgp{a}$.
  Thus $w_1 \equiv t^{-h}a^{\ell}$, for some $\ell \in \Z$, and so
  $w_2 \equiv w_1^{-1}w \equiv a^{-\ell}t^h t^{-h}a^k \equiv a^{k -
    \ell}$.  If $w$ had a prefix $w'$ with $\mu(w') < -h$ then $w'$
  would have to be longer than $w_1$ and so we would have
  $w' = w_1w_2'$.  Then $w_2'$ would be a prefix of $w_2$ of height
  $\mu(w_2') = \mu(w') - \mu(w_1) < 0$, which by
  \Lemitmref{notgeo}{downup} would contradict the geodesicity of $w$.

  We now prove that $w = xy$ such that $x$ is ascending and $y$ is
  descending.  If $w$ has no such decomposition then $w$ has a subword
  of the form $t^{-1}w't$.  An innermost such subword has the form
  $t^{-1}a^kt$.  By \Lemitmref{notgeo}{tiak}, we have $|k| \le 1$.
  So, since $t^{-1}t$ is not geodesic, we have
  $w = w_1t^{-1}a^{\epsilon}tw_2$ with $\epsilon = \pm 1$.  We have
  $\mu(w_1t^{-1}a^{\epsilon}) \ge -h$ and so
  $\mu(w_1t^{-1}a^{\epsilon}t) > -h$.  So $\mu(w_2) < 0$ and the
  shortest prefix of $w_2$ of negative height has the form
  $w_2't^{-1}$ with $\mu(w_2') = 0$.  Then, by \Lemref{zeroheight}, we
  have $w_2' \equiv a^k$ for some $k$.  But then
  \[ w_1t^{-1}a^{\epsilon}tw_2't^{-1} \equiv
    w_1t^{-1}tw_2't^{-1}a^{\epsilon} \equiv
    w_1w_2't^{-1}a^{\epsilon} \] and so
  $w_1t^{-1}a^{\epsilon}tw_2't^{-1}$ is a nongeodesic subword of $w$,
  which is a contradiction.
\end{proof}

\begin{lem}
  \lemlabel{notgoingdown} Let $h \ge 1$, let $k \ge 2^h$ and let
  $\epsilon = \pm 1$.  Let $w$ be a geodesic word with
  $w \equiv t^{-h}a^{\epsilon k}$.  Then the first letter of $w$ is
  not $t^{-1}$.
\end{lem}
\begin{proof}
  Suppose the first letter of $w$ is $t^{-1}$.  Then, by
  \Lemref{descent}, we see that $w$ is descending.  Hence
  \[ w = t^{-\ell_1}a^{k_1}t^{-\ell_2}a^{k_2}\cdots
    t^{-\ell_m}a^{k_m} \] with $\sum_i \ell_i = h$ and $\ell_i > 0$
  for all $i$.  So we have
  \[ w \equiv t^{-\sum_i \ell_i}a^{2^{L_1}k_1 + 2^{L_2}k_2 + \cdots +
      2^{L_m}k_m} \] where $L_j = \sum_{i > j}\ell_i$ and so
  $\epsilon k = 2^{L_1}k_1 + 2^{L_2}k_2 + \cdots + 2^{L_m}k_m$.  But,
  by \Lemitmref{notgeo}{tiak}, we have $|k_i| \le 1$ for all $i$ and
  so \[ |k| \le 2^{L_1} + 2^{L_2} + \cdots + 2^{L_m} \] with
  \[ 0 = L_m < L_{m-1} < \cdots < L_1 < h \] which implies
  $|k| \le \sum_{j=0}^{h-1}2^j = 2^h - 1$, a contradiction.
\end{proof}

\begin{lem}
  \lemlabel{goingdown} Let $h \ge 1$, let $0 \le k \le 2^h$ and let
  $\epsilon = \pm 1$.  Let $w$ be a geodesic word with
  $w \equiv t^{-h}a^{\epsilon k}$.  Then $w$ is descending and every
  prefix $w'$ of $w$ satisfies $w' \equiv t^{-h'}a^{\epsilon k'}$
  where $0 \le h' \le h$ and $0 \le k' \le 2^{h'}$.
\end{lem}
\begin{proof}
  Since there is an automorphism of $\BS(1,2)$ fixing $t$ and sending
  $a$ to $a^{-1}$, we may assume that $\epsilon = 1$.  The proof that
  $w$ is descending is by induction on the length of $w$.  If
  $|w| = 1$ then $w = t^{-1}$ and so satisfies the required
  conditions.  Assume now that $|w| > 1$.  Consider the path
  $f \colon P \to T$ followed by $w$ in the Bass-Serre tree $T$.  Let
  $v_1$ and $v_2$ be the initial and final vertices of this path.  The
  shortest path in $T$ from $v_1$ to $v_2$ is labeled $t^{-h}$.  By
  \Lemitmref{notgeo}{downup}, the path $f$ may not traverse an edge
  below $v_2$.  Hence, any instance of $t$ in $w$ corresponds to an
  edge of $T$ which is ascended by $f$ and later descended.  That is,
  the instance of $t$ is the first letter of a subword $tw't^{-1}$ of
  $w$ with $w' \equiv a^{k'}$ for some $k'$.  Then, if $w$ has an
  instance of $t$ then, by \Lemitmref{notgeo}{downupdown}, it must
  occur to the left of any negative power of $t$.  So
  $w = a^{\eta}tw'$ for some $\eta \in \{-1,0,1\}$ and some word $w'$.
  But then
  \[ w' \equiv t^{-1}a^{-\eta}w \equiv
    t^{-1}a^{-\eta}t^{-h}a^{k} \equiv
    t^{-(h+1)}a^{-2^h\eta+k} \] and
  $|-2^h\eta + k| \le 2^h + 2^h = 2^{h+1}$ and so, by
  induction, $w'$ must have a prefix of the form $t^{-1}$ or
  $a^{\epsilon'}t^{-1}$, where $\epsilon'$ is the sign of
  $-2^h\eta + k$.  But then $w$ must contain a subword
  $tt^{-1}$ or $ta^{\epsilon'}t^{-1}$, which are not geodesic.  So we
  see that $w$ is descending.

  It remains to show that every prefix $w'$ of $w$ satisfies the
  condition $(\ast)$ that $w' \equiv t^{-h'}a^{k'}$ where
  $0 \le h' \le h$ and $0 \le k' \le 2^{h'}$.  That
  $w' \equiv t^{-h'}a^{k'}$ with $0 \le h' \le h$ holds because $w$ is
  descending and $f$ does not descend below $v_2$ in $T$.  Assume for
  the sake of finding a contradiction that $w$ does not satisfy
  $(\ast)$.  Note that $w' \equiv t^{-h'}a^{k'}$ satisfies $0 \le k'$
  (respectively $k' \le 2^{h'}$) if and only if
  $w't^{-1} \equiv t^{-h'}a^{k'}t^{-1} \equiv t^{-(h'+1)}a^{2k'}$
  satisfies $0 \le 2k'$ (respectively $2k' \le 2^{h'+1}$).  So the
  shortest prefix of $w$ that violates $(\ast)$ has the form
  $w'a^{\epsilon}$ and the shortest prefix of $w$ of length at least
  $|w'|$ that satisfies $(\ast)$ has the form
  $w'a^{\epsilon}w''a^{-\epsilon}$, where for some $h', h'' \ge 0$,
  either $\epsilon = -1$ and $w' \equiv t^{-h'}$ and
  $w'a^{\epsilon}w''a^{-\epsilon} \equiv t^{-(h'+h'')}$ or
  $\epsilon = 1$ and $w' \equiv t^{-h'}a^{2^{h'}}$ and
  $w'a^{\epsilon}w''a^{-\epsilon} \equiv t^{-(h'+h'')}a^{2^{h'+h''}}$.
  We will prove that $a^{\epsilon}w''a^{-\epsilon} \equiv t^{-h''}$,
  which contradicts geodesicity of $w$ by \Lemitmref{notgeo}{tconvex}.
  In the case $\epsilon = -1$, we have
  \[ a^{\epsilon}w''a^{-\epsilon} \equiv
    (w')^{-1}w'a^{\epsilon}w''a^{-\epsilon} \equiv t^{h'}
    t^{-(h'+h'')} \equiv t^{-h''} \] and, in the case $\epsilon = 1$,
  we have
  \begin{align*}
    a^{\epsilon}w''a^{-\epsilon}
    &\equiv (w')^{-1}w'a^{\epsilon}w''a^{-\epsilon} \\
    &\equiv a^{-2^{h'}}t^{h'} t^{-(h'+h'')}a^{2^{h'+h''}} \\
    &\equiv a^{-2^{h'}} t^{-h''}a^{2^{h'+h''}} \\
    &\equiv t^{-h''}
  \end{align*}
  so we have our contradiction in either case.
\end{proof}

\begin{lem}
  \lemlabel{drift} Let $h \ge 1$, let $0 \le k \le 2^h$ and let
  $\epsilon = \pm 1$.  The following statements describe precisely
  which initial letters a geodesic word
  $w \equiv t^{-h}a^{\epsilon k}$ may have.
  \begin{enumerate}
  \item \itmlabel{small} If $k < \bigl(\frac{2}{3}\bigr)2^h$ then any
    geodesic word $w \equiv t^{-h}a^{\epsilon k}$ has the form
    $w = t^{-1}w'$.
  \item \itmlabel{med} If
    $\bigl(\frac{2}{3}\bigr)2^h < k < \bigl(\frac{5}{6}\bigr)2^h$ then
    any geodesic word $w \equiv t^{-h}a^{\epsilon k}$ has the form
    $w = t^{-1}w'$ or $w = a^{\epsilon}w''$ and there exist geodesics
    of both forms.
  \item \itmlabel{big} If $k > \bigl(\frac{5}{6}\bigr)2^h$ then any
    geodesic word $w \equiv t^{-h}a^{\epsilon k}$ has the form
    $w = a^{\epsilon}w'$.
  \end{enumerate}
\end{lem}
\begin{proof}
  The proof is by induction on $h$.  If $h = 1$ then
  $k \in \{0, 1, 2\}$.  If $k = 0$ then
  $k < \frac{4}{3} = \bigl(\frac{2}{3}\bigr)2^h$ and the only geodesic
  word $w \equiv t^{-h}a^{\epsilon k}$ is $t^{-1}$.  If $k = 1$ then
  $k < \frac{4}{3} = \bigl(\frac{2}{3}\bigr)2^h$ and the only geodesic
  word $w \equiv t^{-h}a^{\epsilon k}$ is $t^{-1}a^{\epsilon}$.  If
  $k = 2$ then $k > \frac{5}{3} = \bigl(\frac{5}{6}\bigr)2^h$ and the
  only geodesic word $w \equiv t^{-h}a^{\epsilon k}$ is
  $a^{\epsilon}t^{-1}$.  So, in all cases, the lemma holds for $h = 1$.

  If $h = 2$ then $k \in \{ 0,1,2,3,4 \}$.  By \Lemref{goingdown}, we
  need only consider descending words whose prefixes are equivalent to
  \[ t^{-h'}a^{k'\epsilon} \] for $h' \in \{ 0,1,2 \}$ and
  $k' \in \{0, 1, \ldots, 2^{h'} \}$.  We may also exclude words with
  backtracks and, by \Lemitmref{notgeo}{tiak}, those containing
  $t^{-1}a^{\ell}$ with $|\ell| \ge 2$ and, by
  \Lemitmref{notgeo}{rightdownleft}, those containing
  $a^{\epsilon}t^{-1}a^{-\epsilon}$.  Then the list of all possible
  geodesics is
  \begin{align*}
    t^{-2}a^{\epsilon \ell} & \\
    t^{-1}a^{\epsilon}t^{-1}a^{\epsilon \ell} &\equiv t^{-2}a^{\epsilon(2 + \ell)} \\
    a^{\epsilon}t^{-2}a^{-\epsilon \ell} &\equiv t^{-2}a^{\epsilon(4 - \ell)}
  \end{align*}
  with $\ell \in \{0,1\}$.  Only one pair of the words in this list,
  namely $a^{\epsilon}t^{-2}a^{-\epsilon}$ and
  $t^{-1}a^{\epsilon}t^{-1}a^{\epsilon}$, are equivalent and they have
  the same length.  Hence the list is exactly the list of all
  geodesics equivalent to $t^{-2}a^{\epsilon k}$ with
  $k \in \{0,1,2,3,4\}$.  Now, if $k < \bigl(\frac{2}{3}\bigr)2^h$
  then $k \in \{0,1,2\}$.  All geodesics equivalent to
  $t^{-2}a^{\epsilon k}$ with $k \in \{0,1,2\}$ are of the first two
  forms in the list which have initial letter $t^{-1}$.  The only $k$
  with $\bigl(\frac{2}{3}\bigr)2^h < k < \bigl(\frac{5}{6}\bigr)2^h$
  is $k = 3$ and $t^{-2}a^{\epsilon 3}$ is equivalent, with $\ell$ set
  to $1$, to both the second form and the third form, which have
  initial letter $t^{-1}$ and $a^{\epsilon}$.  Finally, if
  $k > \bigl(\frac{5}{6}\bigr)2^h$ then $k = 4$ which is equivalent
  only to the last form in the list with $\ell = 0$ and this form has
  initial letter $a^{\epsilon}$.  So we see that the lemma holds for
  $h = 2$.  Going forward we assume that $h > 2$.
  
  Suppose $k < \bigl(\frac{2}{3}\bigr)2^h$.  To show that a geodesic
  $w \equiv t^{-h}a^{\epsilon k}$ has initial letter $t^{-1}$ it
  suffices, by \Lemref{goingdown}, to rule out the possibility that
  $w$ has the form $a^{\epsilon}t^{-1}w'$.  If that were the case
  then, by \Lemref{goingdown}, the first letter of $w'$ would be
  either $t^{-1}$ or $a^{-\epsilon}$ and we would have
  $w' \equiv ta^{-\epsilon}t^{-h}a^{\epsilon k} \equiv
  t^{-(h-1)}a^{-\epsilon(2^h - k)}$.  If $2^h - k > 2^{h-1}$ then, by
  \Lemref{notgoingdown}, the first letter of $w'$ is not $t^{-1}$ and
  so would have to be $a^{-\epsilon}$. But
  $a^{\epsilon}t^{-1}a^{-\epsilon}$ is not geodesic, by
  \Lemitmref{notgeo}{rightdownleft}.  So we have
  $2^{h-1} \ge 2^h - k > \bigl(\frac{2}{3}\bigr)2^{h-1}$.  Then,
  applying (\itmref{med}) and (\itmref{big}) inductively to $h-1$ and
  $2^h - k$ and $-\epsilon$, we see that $w'$ is equivalent to a
  geodesic of the form $a^{-\epsilon}w''$.  But then
  $a^{\epsilon}t^{-1}a^{-\epsilon}w''$ is geodesic and this cannot be
  by \Lemitmref{notgeo}{rightdownleft}.

  Suppose $k > \bigl(\frac{5}{6}\bigr)2^h$.  Let $w$ be a geodesic
  with $w \equiv t^{-h}a^{\epsilon k}$.  Suppose the initial letter of
  $w$ is not $a^{\epsilon}$.  Then, by \Lemref{goingdown}, the first
  letter of $w$ is $t^{-1}$ and the second letter of $w$ is either
  $t^{-1}$ or $a^{\epsilon}$.  By \Lemref{notgoingdown}, the second
  letter of $w$ cannot be $t^{-1}$ and so we have
  $w = t^{-1}a^{\epsilon}w'$ for some $w'$.  Hence
  $w' \equiv a^{-\epsilon}tt^{-h}a^{\epsilon k} \equiv
  t^{-(h-1)}a^{\epsilon(k - 2^{h-1})}$ with $k - 2^{h-1} \le 2^{h-1}$
  and
  $k - 2^{h-1} > \bigl(\frac{5}{6}\bigr)2^h - 2^{h-1} =
  \bigl(\frac{2}{3}\bigr)2^{h-1}$.  So, inductively applying
  (\itmref{med}) and (\itmref{big}), we see that $w'$ can be replaced
  by a geodesic of the form $a^{\epsilon}w''$.  But then
  $t^{-1}a^{\epsilon}a^{\epsilon}w''$ is a geodesic, which is a
  contradiction.

  Suppose
  $\bigl(\frac{2}{3}\bigr)2^h < k < \bigl(\frac{5}{6}\bigr)2^h$.  That
  any geodesic $w \equiv t^{-h}a^{\epsilon k}$ has the form
  $w = t^{-1}w'$ or $w = a^{\epsilon}w''$ follows from
  \Lemref{goingdown}.  Consider first the case where we have a
  geodesic of the form $w = t^{-1}w'$.  Then, by \Lemref{goingdown},
  the initial letter of $w'$ is either $t^{-1}$ or $a^{\epsilon}$.
  But $w' \equiv tw \equiv t^{-(h-1)}a^{\epsilon k}$ with
  $k > \bigl(\frac{2}{3}\bigr)2^h > 2^{h-1}$ and so, by
  \Lemref{notgoingdown}, the initial letter of $w'$ is $a^{\epsilon}$.
  So $w' = a^{\epsilon}u'$ with
  $u' \equiv a^{-\epsilon}w' \equiv t^{-(h-1)}a^{\epsilon(k-2^{h-1})}$
  and
  $k - 2^{h-1} < \bigl(\frac{5}{6}\bigr)2^h - 2^{h-1} =
  \bigl(\frac{2}{3}\bigr)2^{h-1}$.  So, by induction, we have
  $u' = t^{-1}x'$ with $x' \equiv t^{-(h-2)}a^{\epsilon(k-2^{h-1})}$
  and
  $k-2^{h-1} > \bigl(\frac{2}{3}\bigr)2^h - 2^{h-1} =
  \bigl(\frac{2}{3}\bigr)2^{h-2}$.  Then, either by induction if
  $k \le 2^{k-2}$ or otherwise by \Lemref{notgoingdown}, we have that
  $x'$ is equivalent to a geodesic of the form $a^{\epsilon}y'$.  Thus
  we have a geodesic
  $t^{-1}a^{\epsilon}t^{-1}a^{\epsilon}y' \equiv t^{-h}a^{\epsilon
    k}$.  But $a^{\epsilon}t^{-2}a^{-\epsilon}$ has the same length as
  $t^{-1}a^{\epsilon}t^{-1}a^{\epsilon}$ and
  $a^{\epsilon}t^{-2}a^{-\epsilon} \equiv
  t^{-1}a^{\epsilon}t^{-1}a^{\epsilon}$ and so
  $a^{\epsilon}t^{-2}a^{-\epsilon}y'$ is a geodesic with
  $a^{\epsilon}t^{-2}a^{-\epsilon}y' \equiv t^{-h}a^{\epsilon k}$.
  Now, consider the case where we have a geodesic of the form
  $w = a^{\epsilon}w''$.  By \Lemref{goingdown}, the next two letters
  of $w''$ are either $t^{-2}$ or $t^{-1}a^{-\epsilon}$ but
  $a^{\epsilon}t^{-1}a^{-\epsilon}$ is not geodesic, by
  \Lemitmref{notgeo}{rightdownleft}, so we must have
  $w = a^{\epsilon}t^{-2}x''$.  Then
  $x'' \equiv t^2a^{-\epsilon}t^{-h}a^{\epsilon k} \equiv
  t^{-(h-2)}a^{-\epsilon(2^h - k)}$ with
  $2^h - k > 2^h - \bigl(\frac{5}{6}\bigr)2^h =
  \bigl(\frac{2}{3}\bigr)2^{h-2}$.  Then, either by induction if
  $2^h - k \le 2^{k-2}$ or otherwise by \Lemref{notgoingdown}, we have
  that $x''$ is equivalent to a geodesic of the form
  $a^{-\epsilon}y''$.  So we have a geodesic
  $a^{\epsilon}t^{-2}a^{-\epsilon}y'' \equiv t^{-h}a^k$ and we may
  replace $a^{\epsilon}t^{-2}a^{-\epsilon}$ with
  $t^{-1}a^{\epsilon}t^{-1}a^{\epsilon}$ to obtain an equivalent
  geodesic $t^{-1}a^{\epsilon}t^{-1}a^{\epsilon}y''$.
\end{proof}

\begin{lem}
  \lemlabel{akgeos} Let $w$ be a geodesic with $w \equiv a^k$.  Then
  $w = xa^{\ell}y$ such that the following conditions hold.
  \begin{enumerate}
  \item $x$ is ascending and does not have terminal letter $a^{\pm 1}$.
  \item $y$ is descending and does not have initial letter $a^{\pm 1}$.
  \item $\ell$ has the same sign as $k$.
  \item
    $-\bigl(\frac{2}{3}\bigr)2^h < |k| - 2^h|\ell| <
    \bigl(\frac{5}{3}\bigr)2^h$ where $0 \le h = \mu(x) = -\mu(y)$.
  \end{enumerate}
\end{lem}
\begin{proof}
  By \Lemref{descent}, we have $w = xa^{\ell}y$ with $x$ ascending and
  $y$ descending.  Choosing $x$ and $y$ so as to maximize the length
  of the $a^{\ell}$ part ensures that the terminal letter of $x$ is
  $t$ and the initial letter of $y_i$ is $t^{-1}$.  Since $\mu(w) = 0$
  and $\mu(a^{\ell}) = 0$, we have $\mu(x) + \mu(y) = \mu(w) = 0$.

  If $h = 0$ then $w = a^k = a^{\ell}$ and so the remaining conditions
  hold.  So let $h \ge 1$.  Then, by \Lemitmref{notgeo}{tati}, we have
  $|\ell| \ge 2$.  Let $\epsilon$ be the sign of $\ell$.  Since $y$ is
  descending, we have $y \equiv t^{-h}a^{\epsilon m}$ for some
  $m \in \Z$.  Since the first letter of $y$ is $t^{-1}$, by
  \Lemref{notgoingdown}, we have $|m| \le 2^h$.  Then, by
  \Lemref{drift}, we have $|m| < \bigl(\frac{5}{6}\bigr)2^h$.  Now
  $a^{\epsilon}y$ is a subword of $w$ and
  $a^{\epsilon}y \equiv a^{\epsilon}t^{-h}a^{\epsilon m} \equiv
  t^{-h}a^{\epsilon (m+2^h)}$ with
  $m + 2^h > 2^h - \bigl(\frac{5}{6}\bigr)2^h > 0$.  Then, either
  $m \ge 0$ so that $m + 2^h \ge 2^h > \bigl(\frac{2}{3}\bigr)2^h$ or
  $m < 0$ so that $2^h > m + 2^h > 0$ and, by \Lemref{drift}, we have
  $m + 2^h > \bigl(\frac{2}{3}\bigr)2^h$.  Hence we have
  $-\bigl(\frac{1}{3}\bigr)2^h < m < \bigl(\frac{5}{6}\bigr)2^h$.
  Applying the exact same arguments to the subwords $x^{-1}$ and
  $a^{-\epsilon}x^{-1}$ of $w^{-1} = y^{-1}a^{-\ell}x^{-1}$ we see
  that $x^{-1} \equiv t^{-h}a^{-\epsilon n}$ with
  $-\bigl(\frac{1}{3}\bigr)2^h < n < \bigl(\frac{5}{6}\bigr)2^h$.
  Hence
  \[ w \equiv a^{\epsilon n}t^h a^{\ell} t^{-h}a^{\epsilon m} \equiv
    a^{\epsilon(n+m + 2^h|\ell|)} \] and so
  $k = \epsilon(n+m + 2^h|\ell|)$ which gives
  \[ -\biggl(\frac{2}{3}\biggr)2^h + 2^h|\ell| < \epsilon k <
    \biggl(\frac{5}{3}\biggr)2^h + 2^h|\ell|\] and so, as $|\ell| \ge 2$
  we see that $\epsilon k > 0$.  Then $k$ has the same sign as $\ell$
  so $\epsilon k = |k|$ and we have
  \[ -\biggl(\frac{2}{3}\biggr)2^h < |k| - 2^h|\ell| <
    \biggl(\frac{5}{3}\biggr)2^h \] as required.
\end{proof}

\subsection{Isometric cycles in \texorpdfstring{$\BS(1,2)$}{BS(1,2)}}

Let $f \colon C \to \Gamma$ be a cycle in $\Gamma$.  The edges of $C$
are naturally directed and labeled by the generators $a$ and $t$ by
pulling back from $\Gamma$.  In this way, each path $P \to C$ is
labeled by a word in $a$, $t$ and their inverses.

A \defterm{traversal} of a cycle $C$ is an immersed path $P \to C$
such that $|P| = |C|$.  A traversal of $C$ is equivalent to a choice
of basepoint and orientation of $C$.  Note that if
$f \colon C \to \Gamma$ is a cycle in $\Gamma$ then the label of any
traversal of $C$ determines $f$ up to translation by elements of
$\BS(1,2)$.

\begin{lem}
  \lemlabel{isomcycles} Let $f \colon C \to \Gamma$ be an
  isometrically embedded cycle of length $|C| > 5$.  Then some
  traversal of $C$ is labeled by a word of the form
  \[ w = a^{\epsilon_1}w_1a^{\epsilon_2}w_2 \] with
  $\epsilon_1,\epsilon_2 \in \{\pm 1\}$ such that $w_i$ satisfies the
  following properties for each $i$.
  \begin{enumerate}
  \item $w_i$ is geodesic with initial letter $t$ and terminal letter
    $t^{-1}$.
  \item $w_i \equiv a^{2k_i}$ with $|k_1 + k_2| \le 1$.
  \end{enumerate}
\end{lem}
\begin{proof}
  Let $f \colon C \to \Gamma$ be an isometrically embedded cycle.  Let
  $\bar f$ be the projection of $f$ to the Bass-Serre tree $T$.  By
  \Lemref{minheight}, there is a unique vertex $v$ of minimal height
  of $\bar f(C)$.  Since the preimage of $v$ in $\Gamma$ has no
  embedded cycles, we cannot have $\bar f(C) = \{v\}$.  Then $C$ must
  contain a subpath labeled $t^{-1}a^kt$ with the $a^k$ part mapping
  to $v$ under $\bar f$.  Since $f$ is an isometric embedding and
  $|C|> 5$, any subpath of $f$ of length three is geodesic.  Then
  \Lemitmref{notgeo}{tiak} implies that $|k| = 1$.  So
  $t^{-1}a^kt=t^{-1}a^{\pm 1}t$ corresponds to a nonbacktracking path
  in $T$.  It follows, since $f$ is a closed path and $v$ is a
  cutpoint of $T$, that $C$ contains another subpath labeled
  $t^{-1}a^{\pm 1}t$ such that $\bar f$ sends $a^{\pm 1}$ to $v$.  So
  some traversal of $C$ is labeled by a word
  \[ w = a^{\epsilon_1}tu_1t^{-1}a^{\epsilon_2}tu_2t^{-1} \] with
  $\epsilon_1,\epsilon_2 \in \{\pm 1\}$, such that
  $u_i \equiv a^{k_i}$ for some $k_1,k_2 \in \Z$.  Let
  $w_i = tu_it^{-1}$.  Then $w_i \equiv a^{2k_i}$ so we have
  \[ \epsilon_1 + 2k_1 + \epsilon_2 + 2k_2 = 0\] since $w$ is trivial
  in $G$, and this implies $|k_1 + k_2| \le 1$.  Also we have
  \[ a^{\epsilon_i}w_ia^{\epsilon_{i+1}} \equiv
    w_ia^{\epsilon_i}a^{\epsilon_{i+1}} \] which shows that
  $a^{\epsilon_i}w_ia^{\epsilon_{i+1}}$ has the same length and
  endpoints as $w_ia^{\epsilon_i}a^{\epsilon_{i+1}}$.  But
  $w_ia^{\epsilon_i}a^{\epsilon_{i+1}}$ cannot be geodesic since it
  either backtracks or contains $t^{-1}a^k$ with $|k| = 2$ and so
  $a^{\epsilon_i}w_ia^{\epsilon_{i+1}}$ is not geodesic either.
  Hence, since $f$ is an isometric embedding, the complementary path
  $w_{i+1}$ is geodesic.
\end{proof}

\begin{thm}
  \thmlabel{bss} The standard Cayley graph of of $\BS(1,2)$ is
  shortcut.
\end{thm}
\begin{proof}
  Let $\Gamma$ be the standard Cayley graph of $\BS(1,2)$.  We will
  show that there are no isometrically embedded cycles
  $f \colon C \to \Gamma$ of length $|C| > 5$.  For the sake of
  deriving a contradiction, suppose $f$ is such a cycle.  Then, by
  \Lemref{isomcycles}, some traversal of $C$ is labeled
  \[ w = a^{\epsilon_1}w_1a^{\epsilon_2}w_2 \] such that $w_i$ is
  geodesic with initial letter $t$ and terminal letter $t^{-1}$ and
  $w_i \equiv a^{2k_i}$ with $|k_1 + k_2| \le 1$.  Then, by \lemref{akgeos}, we
  have \[ w_i = x_ia^{\ell_i}y_i \] where $x_i$ is ascending and has
  initial and terminal letter $t$, where $y_i$ is descending and has
  initial and terminal letter $t^{-1}$, where $\ell_i$ has the same
  sign as $k_i$ and where
  \[ -\biggl(\frac{2}{3}\biggr)2^{h_i} < |k_i| - 2^{h_i}|\ell_i| <
    \biggl(\frac{5}{3}\biggr)2^{h_i} \] with
  $0 \le h_i = \mu(x_i) = -\mu(y_i)$.  Since $x_i$ has terminal letter
  $t$ and $y_i$ has initial letter $t^{-1}$, we have $h_i \ge 1$ and,
  by \Lemitmref{notgeo}{tati}, we have $|\ell_i| \ge 2$.

  We may assume that $h_1 \le h_2$ since otherwise we may replace $w$
  with a cyclic permutation.  We must have $|k_i| \ge 1$ since
  otherwise $w_i = 1$ which does not start with $t$.  Hence, as
  $|k_1 + k_2| \le 1$ we must have that $k_1$ and $k_2$ have opposite
  signs.  Since there is an automorphism of $\BS(1,2)$ fixing $t$ and
  mapping $a \mapsto a^{-1}$ we may assume that $k_1 > 0$ and
  $k_2 < 0$.  Then $\ell_1 > 0$ and $\ell_2 < 0$.

  Let $p_i \in C$ be the midpoint of the subpath $a^{\ell_i}$ of $w$.
  Abusing notation we write the two segments of $C$ between $p_1$ and
  $p_2$ as $a^{\frac{1}{2}\ell_1}y_1a^{\epsilon_2}x_2a^{\frac{1}{2}\ell_2}$ and
  $a^{\frac{1}{2}\ell_2}y_2a^{\epsilon_1}x_1a^{\frac{1}{2}\ell_1}$ which may be thought
  of as combinatorial paths in the barycentric subdivision of
  $\Gamma$.  Since $f$ is an isometric embedding, one of these two
  paths must be geodesic in $\Gamma$.  We have
  \[ y_1a^{\epsilon_2}x_2a^{-1} \equiv y_1a^{\epsilon_2}a^{-2^{h_2}}x_2
    \equiv a^{-2^{h_2-h_1}}y_1a^{\epsilon_2}x_2\]
  \[ a^{-1}y_2a^{\epsilon_1}x_1 \equiv
    y_2a^{-2^{h_2}}a^{\epsilon_1}x_1 \equiv
    y_2a^{\epsilon_1}x_1a^{-2^{h_2-h_1}} \] and so
  \[ a^{\frac{1}{2}\ell_1}y_1a^{\epsilon_2}x_2a^{\frac{1}{2}\ell_2} \equiv
    a^{\frac{1}{2}\ell_1 - 2^{h_2-h_1}}y_1a^{\epsilon_2}x_2a^{\frac{1}{2}\ell_2+1}\]
  \[ a^{\frac{1}{2}\ell_2}y_2a^{\epsilon_1}x_1a^{\frac{1}{2}\ell_1} \equiv
    a^{\frac{1}{2}\ell_2+1}y_2a^{\epsilon_1}x_1a^{\frac{1}{2}\ell_1 - 2^{h_2-h_1}}\]
  which, by geodesicity of one of these paths, imply that either
  \[ \Bigl|\frac{1}{2}\ell_1\Bigr| + \Bigl|\frac{1}{2}\ell_2\Bigr| \le \Bigl|\frac{1}{2}\ell_1-2^{h_2-h_1}\Bigr| +
    \Bigl|\frac{1}{2}\ell_2+1\Bigr| \] or
  \[ \Bigl|\frac{1}{2}\ell_2\Bigr| + \Bigl|\frac{1}{2}\ell_1\Bigr| \le \Bigl|\frac{1}{2}\ell_2+1\Bigr| +
    \Bigl|\frac{1}{2}\ell_1-2^{h_2-h_1}\Bigr| \] but these two inequalities are equivalent
  so they must both hold.  Then, using $\ell_1 \ge 0$ and
  $\ell_2 \le -2$, we obtain
  $\ell_1 + 2 \le |\ell_1 - 2^{h_2-h_1+1}|$.  Since $\ell_1 \ge 0$,
  this inequality may only hold if $2^{h_2-h_1+1} > \ell_1$ and so we
  have $\ell_1 + 1 \le 2^{h_2-h_1}$.  Now $\ell_1 \ge 2$ and so we
  have $h_2 \ge h_1 + 2$.  The following computation makes use of
  various inequalities which have been established thus far in this
  proof.
  \begin{align*}
    \biggl(\frac{2}{3}\biggl)2^{h_2-2} + 2^{h_2}
    &\ge \biggl(\frac{2}{3}\biggr)2^{h_1} + 2^{h_2-h_1}\cdot 2^{h_1} \\
    &\ge \biggl(\frac{2}{3}\biggr)2^{h_1} + (\ell_1 + 1)2^{h_1} \\
    &= \biggl(\frac{5}{3}\biggr)2^{h_1} + 2^{h_1}|\ell_1| \\
    &> |k_1|
      \ge |k_2| - 1
      > -\biggl(\frac{2}{3}\biggr)2^{h_2} + 2^{h_2}|\ell_2| - 1
  \end{align*}
  So, as $|\ell_2| \ge 2$, we have
  \[ \biggl(\frac{2}{3}\biggl)2^{h_2-2} + 2^{h_2} >
    -\biggl(\frac{2}{3}\biggr)2^{h_2} + 2^{h_2}\cdot 2 - 1 \] which we
  manipulate to obtain the equivalent inequality $2^{h_2} < 6$.  Then
  $h_2 \le 2$, which implies that $h_1 = 0$.  But this is a
  contradiction as we established above that $h_1 \ge 1$.
\end{proof}

\subsection{A Cayley graph of \texorpdfstring{$\BS(1,2)$}{BS(1,2)}
  that is not shortcut}

We now turn our attention to a different generating set of $\BS(1,2)$.
Let $\Gamma$ be the Cayley graph of $\BS(1,2)$ with the generating set
$\{a,t,\tau\}$ where $\tau = t^2$.

\begin{lem}
  \lemlabel{powtsum} Let $k \ge 1$ and let $0 \le z_{\max} \le k$.
  Suppose
  \[ \sum_{z = -m}^{z_{\max}} \alpha_z 2^{z} = 2^k \pm 1 \] where
  $\alpha_z \in \Z$ and $m \ge 0$.
  \begin{itemize}
  \item If $z_{\max} = 0$ then
    $\sum_z |\alpha_z| \ge 2^{k - z_{\max}} - 1$
  \item If $z_{\max} = 1$ then
    $\sum_z |\alpha_z| \ge 2^{k - z_{\max}}$
  \item If $z_{\max} \ge 2$ then
    $\sum_z |\alpha_z| \ge 2^{k - z_{\max}} + 1$
  \end{itemize}
\end{lem}
\begin{proof}
  If $z_{\max} = 0$ then
  \[ 2^k - 1 \le 2^k \pm 1 = \Bigl|\sum_{z = -m}^{z_{\max}}
    \alpha_z 2^{z}\Bigr| \le \sum_{z = -m}^{z_{\max}} |\alpha_z|
    2^{z_{\max}} = \sum_{z = -m}^{z_{\max}} |\alpha_z| \]
  and so we have $\sum_z |\alpha_z| \ge 2^{k - z_{\max}} - 1$.

  Suppose $z_{\max} = 1$ and $\sum_z |\alpha_z| < 2^{k - z_{\max}}$.
  Then $\sum_z |\alpha_z| \le 2^{k - z_{\max}} - 1$ and so
  \[ \biggl|\sum_{z = -m}^{z_{\max}} \alpha_z 2^{z}\biggr| \le \sum_{z
      = -m}^{z_{\max}} |\alpha_z| 2^{z_{\max}} \le (2^{k - z_{\max}}
    - 1) \cdot 2^{z_{\max}} = 2^k - 2 \] which is a contradiction.
  So we have $\sum_z |\alpha_z| \ge 2^{k - z_{\max}}$.

  Now, suppose $z_{\max} \ge 2$.  Among all $m \ge 0$ and
  $(\alpha_z)_z$ that satisfy
  \[ \sum_{z = -m}^{z_{\max}} \alpha_z 2^{z} = 2^k \pm 1 \] choose an
  $m \ge 0$ and $(\alpha_z)_z$ that minimizes $\sum_z |\alpha_z|$.  We
  will show that $\sum_z |\alpha_z| \ge 2^{k - z_{\max}} + 1$.  We
  claim that for $z < z_{\max}$, we have $|\alpha_z| \le 1$.  Indeed,
  if $|\alpha_z| \ge 2$, then we can replace $\alpha_z$ with
  $\alpha_z - \epsilon 2$ and $\alpha_{z+1}$ with
  $\alpha_{z+1} + \epsilon$, where $\epsilon$ is the sign of
  $\alpha_z$.  This reduces $\sum_z |\alpha_z|$ while preserving
  $\sum_z \alpha_z 2^{z}$ and so contradicts minimality of $m$ and
  $(\alpha_z)_z$.  Since $2^k \pm 1$ is not even, there must be some
  $\alpha_z \neq 0$ with $z \le 0$.  So if
  $\alpha_{z_{\max}} \ge 2^{k - z_{\max}}$ then
  $\sum_z |\alpha_z| \ge 2^{k - z_{\max}} + 1$.  So we may assume that
  $\alpha_{z_{\max}} < 2^{k - z_{\max}}$.  Say
  $\alpha_{z_{\max}} = 2^{k - z_{\max}} - \ell$ with $\ell \ge
  1$. Then
  $2^k \pm 1 - \alpha_{z_{\max}}2^{z_{\max}} \ge \ell 2^{z_{\max}} -
  1$ and so
  $\sum_{z = -m}^{z_{\max} - 1} \alpha_z 2^{z} \ge \ell 2^{z_{\max}} -
  1$.  But, since $\alpha_z \le 1$ for $z < z_{\max}$ we have
  \[ \ell 2^{z_{\max}} - 1 \le \sum_{z = -m}^{z_{\max} - 1} \alpha_z
    2^{z} \le \sum_{z = -m}^{z_{\max} - 1} 2^{z} = 2^{z_{\max}} -
    2^{-m} \le 2^{z_{\max}} \] and so
  $(\ell - 1)4 \le (\ell - 1)2^{z_{\max}} \le 1$ which implies
  $\ell = 1$.  Thus if $\alpha_{z'} < 1$ for some $z'$ with
  $0 \le z' < z_{\max}$, then we would have
  \begin{align*}
    2^{z_{\max}} - 1
    &\le \sum_{z = -m}^{z_{\max} - 1} \alpha_z 2^{z} \\
    &= \sum_{z = -m}^{-1} \alpha_z 2^{z}
      + \sum_{\substack{z = 0 \\ z \neq z'}}^{z_{\max} - 1} \alpha_z 2^{z}
    + \alpha_{z'} 2^{z'} \\
    &\le \sum_{z = -m}^{-1} 2^{z}
      + \sum_{\substack{z = 0 \\ z \neq z'}}^{z_{\max} - 1} 2^{z} \\
    &= 1 - 2^{-m} + 2^{z_{\max}} - 1 - 2^{z'} \\
    &< 2^{z_{\max}} - 1
  \end{align*}
  a contradiction.  Hence $\alpha_z = 1$ for $0 \le z < z_{\max}$ so
  that
  $\sum_z |\alpha_z| \ge 2^{k - z_{\max}} - 1 + z_{\max} \ge 2^{k -
    z_{\max}} + 1$ as required.
\end{proof}

\begin{lem}
  \lemlabel{attaugeos} Let $k$ and $\ell$ be nonnegative integers with
  $\ell \le k$ and $k \ge 2$.  Then the word
  \[ w = \tau^{\ell}a\tau^{-k}a^{\pm 1}\tau^{k - \ell} \] describes a
  geodesic in $\Gamma$.
\end{lem}
\begin{proof}
  Consider the bijection
  $\phi \colon \BS(1,2) \to \Z\bigl[\frac{1}{2}\bigr] \times \Z$
  described near the beginning of \Secref{bsgp}.  Then, under this
  bijection, $w$ describes a path from $\bigl(0,2(k - \ell)\bigr)$ to
  $\bigl(4^k \pm 1, 2(k - \ell)\bigr)$.  Indeed, following the
  description of the directed labeled edges we have the following.
  \begin{align*}
    \bigl(0,2(k-\ell)\bigr) &\xrightarrow{\tau^{\ell} = t^{2\ell}} (0,2k) \\
    &\xrightarrow{a} (2^{2k},2k) = (4^k,2k) \\
    &\xrightarrow{\tau^{-k} = t^{-2k}} (4^k,0) \\
    &\xrightarrow{a^{\pm 1}} (4^k \pm 1,0) \\
    &\xrightarrow{\tau^{k-\ell} = t^{2(k-\ell)}} \bigl(4^k \pm 1,2(k-\ell)\bigr)
  \end{align*}
  Consider any path $(r_j,z_j)_{j=0}^m$ from $\bigl(0,2(k - \ell)\bigr)$ to
  $\bigl(4^k \pm 1, 2(k - \ell)\bigr)$ following edges of the Cayley
  graph.  It will suffice to show that $m \ge 2k+2$, since
  $2k+2 = |w|$.  For each $j$, either
  \[ (r_{j+1},z_{j+1}) - (r_j,z_j) = (\pm 2^{z_j},0)\]
  which we call a \defterm{horizontal step} or 
  \[ (r_{j+1},z_{j+1}) - (r_j,z_j) =
    \begin{cases}
      (0,\pm 1) \\
      (0,\pm 2)
    \end{cases} \] which we call a \defterm{vertical step}.  Since
  $4^k\pm 1$ is not a power of $2$, there must be at least two
  horizontal steps in $(r_j,z_j)_j$.  Moreover, since $4^k \pm 1$ is
  not even, we must have $z_j \le 0$ for some $j$.  Let
  $z_{\max} = \max_j z_j$.  By joining consecutive $(r_j,z_j)$ by
  horizontal or vertical segments in
  $\R^2 \supset \Z\bigl[\frac{1}{2}\bigr] \times \Z$ we extend
  $(r_j,z_j)_{j=0}^m$ to a piecewise affine map $f \colon P \to \R^2$
  where $|P| = m$.  The projection $P \to \R$ of this map onto the
  second component captures the vertical behaviour of $f$.  We
  collapse any edges of $P$ that map to points under $P \to \R$ to
  obtain a map $\bar f \colon \bar P \to \R$ where $|\bar P|$ is equal
  to the number of vertical steps of $(r_j,z_j)_j$.  The map $\bar f$
  is $2$-Lipschitz and $0, 2(k - \ell), z_{\max} \in \bar f(\bar P)$
  with $z_{\max} \ge 2(k-\ell)$ and $0 \le 2(k - \ell)$.  Then
  \[ \Bigl|\bar f^{-1}\bigl([k - \ell,\infty)\bigr)\Bigr| \ge
    2\cdot\frac{z_{\max} - 2(k-\ell)}{2} \] and
  \[ \Bigl|\bar f^{-1}\bigl((-\infty, k - \ell]\bigr)\Bigr| \ge
    2\cdot\frac{2(k-\ell)}{2} \] where $\bigl|\bar f^{-1}(I)\bigr|$ is
  the sum of the lengths of all maximal segments of $f^{-1}(I)$. Hence
  $|\bar P| \ge z_{\max}$ and so $(r_j,z_j)_j$ takes at least
  $z_{\max}$ vertical steps.  Then, since there must also be at least
  two horizontal steps, we see that if $z_{\max} \ge 2k$ then
  $m \ge 2k + 2$.  So we may assume that $z_{\max} < 2k$.

  We split into three cases: $z_{\max} = 0$, $z_{\max} = 1$ and
  $z_{\max} \ge 2$.  If $z_{\max} = 0$ then, by \Lemref{powtsum},
  there are at least $2^{2k} - 1$ horizontal steps.  So it suffices to
  show that $2^{2k} - 1 \ge 2k + 2$ but, since $k \ge 2$, this follows
  from the fact that $2^x \ge x+3$ for all $x \ge 4$.

  If $z_{\max} = 1$ then, by \Lemref{powtsum}, there are at least
  $2^{2k-1}$ horizontal steps.  There will also be at least
  $z_{\max} = 1$ vertical step.  But a closed path cannot contain a
  single vertical step so there are at least $2$ vertical steps.  So
  it suffices to show that $2^{2k - 1} + 2 \ge 2k + 2$ but, since
  $k \ge 2$, this follows from the fact that $2^x \ge x+1$ for all
  $x \ge 1$.

  If $z_{\max} \ge 2$ then, by \Lemref{powtsum}, there are at least
  $2^{2k - z_{\max}} + 1$ horizontal steps.  There are also at least
  $z_{\max}$ vertical steps.  So it suffices to show that
  $2^{2k - z_{\max}} + 1 + z_{\max} \ge 2k + 2$ but, since
  $z_{\max} < 2k$, this follows from the fact that $2^x \ge x+1$ for
  all $x \ge 1$.
\end{proof}

\begin{lem}
  \lemlabel{attaugc} The word
  \[ w = a \tau^k a \tau^{-k} a^{-1} \tau^k a^{-1} \tau^{-k} \]
  describes an isometric cycle in $\Gamma$ for all $k \ge 1$.
\end{lem}
\begin{proof}
  The word $w$ has length $4k+4$.  After possible inversion and/or the
  application of the automorphism of $\BS(1,2)$ fixing $t$ and sending
  $a \mapsto a^{-1}$, the cyclic subwords of $w$ of length $2k+2$ are
  all of the form in \Lemref{attaugeos} and so are geodesic.  Hence,
  by \Propref{aadist}, the word $w$ describes an isometric cycle in
  $\Gamma$.
\end{proof}

Then we have the following theorem and we see that the shortcut
property for a Cayley graph is not invariant under a change of
generating set.

\begin{thm}
  \thmlabel{bsns} Let $\Gamma$ be the Cayley graph of $\BS(1,2)$ with
  generating set $\{a,t,\tau\}$ where $\tau = t^2$.  Then $\Gamma$ is
  not shortcut.
\end{thm}

\bibliographystyle{abbrv}
\bibliography{nima}
\end{document}